\documentclass{article}
\usepackage[utf8]{inputenc}
\usepackage[a4paper]{geometry}
\usepackage[dvipsnames]{xcolor}

\usepackage{mathtools}
\usepackage{amssymb}
\usepackage{amsthm}
\usepackage{ dsfont }
\usepackage{dirtytalk}
\usepackage{mathabx}
\usepackage{tikz}
\usepackage{tikz-cd}
\usepackage[all,cmtip]{xy}
\usetikzlibrary{decorations.markings}
\usetikzlibrary{decorations.pathreplacing}
\usetikzlibrary{arrows}
\usetikzlibrary{shapes.misc}
\usetikzlibrary{cd}
\usepackage{hyperref}
\usepackage{faktor}
  \usepackage{stmaryrd}

\newtheorem{lemma}{Lemma}[section]
\newtheorem{corollary}[lemma]{Corollary}
\newtheorem{prop}[lemma]{Proposition}
\newtheorem{thm}[lemma]{Theorem}
\newtheorem*{thm*}{Theorem}
\newtheorem{question}{Question}

\theoremstyle{definition}

\newtheorem{defn}[lemma]{Definition}
\newtheorem*{defn*}{Definition}
\newtheorem*{notation}{Notation}

\newtheorem{rmk}[lemma]{Remark}

\newtheorem{example}[lemma]{Example}

\theoremstyle{remark}

\newcommand{\fontact}{{\mathcal C}}

\def\1{\mathds{1}}

\def\dihed{\mathbb{D}_\infty}

\def\G{G}

\def\cont{\mathrm{cont}}

\def\dhaus{d_{\mathrm{Haus}}}
\def\diam{\mathrm{diam}}
\def\Im{\mathrm{Im}}
\def\Ker{\mathrm{Ker}}

\newcommand*\mc[1]{\mathcal{#1}}
\newcommand*\mf[1]{\mathfrak{#1}}

\newcommand*\gate[2]{\mathfrak{g}_{#1}\left(#2\right)}

\def\nest{\sqsubseteq}

\def\gate{\mf{g}}

\let\temp\epsilon
\let\epsilon\varepsilon
\let\varepsilon\temp

\title{Hierarchical hyperbolicity of hyperbolic-2-decomposable groups}
\author{Bruno Robbio, Davide Spriano}
\date{\today}

\begin{document}

\maketitle

\begin{abstract}
Let \(G\) be a graph of hyperbolic groups with 2-ended edge groups. We show that $G$ is hierarchically
hyperbolic if and only if $G$ has no distorted infinite cyclic subgroup. More precisely, we show that $G$ is hierarchically hyperbolic if and only if $G$
does not contain certain quotients of Baumslag--Solitar groups. As a consequence, we obtain several new results about this class, such as quadratic isoperimetric inequality and finite asymptotic dimension. 
\end{abstract}

\tableofcontents

\section{Introduction}
When trying to understand a complicated group, 
a strategy that often proves to be successful is to reduce it to several, more manageable, groups. This can be achieved in several ways, and for this paper we will consider groups that split as graphs of groups with 2-ended edge groups. For the sake of brevity, if \(P\) is a property of a group, we say that a group is \emph{\(P\)-2-decomposable} if it splits as a graph of groups with 2-ended edge groups and vertex groups satisfying property \(P\).

Considering groups of this form is not a novelty in geometric group theory. An important example is the class of \(\mathbb{Z}\)-2-decomposable groups, also known as \emph{generalized Baumslag--Solitar groups} (GBS groups). Although we will not dive deeply in the theory of GBS groups from a traditional viewpoint, it is worth noting that this class has been extensively studied and shown to be an extremely rich object to analyse from multiple points of view. To name a few, GBS groups have been studied in relation with JSJ decompositions (\cite{ForesterGBS}), quasi-isometries (\cite{RipsSubgroupsSmallCancellation}), automorphisms (\cite{LevittAutGBS}) and  cohomological dimension (\cite{krophollerGBS}). For a general overview of results on GBS groups we refer to the survey by Robinson (\cite{robinsonGBS}). 

This paper focuses on hyperbolic-2-decomposable groups satisfying a technical condition called \emph{balancedness}. A group $G$ is said to be balanced if for every $g\in G$ of infinite order, whenever $hg^ih^{-1}=g^j$ for some $h\in G$ it follows that $|i|=|j|$. The notion of balancedness played an important role in the theory of graphs of groups. In \cite{wise2000subgroup}, the author shows that a free-2-decomposable group is subgroup separable if and only if it is balanced. In \cite{ShepherdWoodhouse:QIrigidity}, the authors extend Wise's result to (virtually-free)-2-decomposable groups, obtaining quasi-isometrical rigidity for certain balanced groups. In \cite{button2015balanced} the author studies the relation between possible acylindrical actions of (torsion-free)-2-decomposable groups in connection with balancedness of such groups.

In our main result (Theorem~\ref{thmi: HHG iff no BS subgroups} below) we show that every balanced hyperbolic-2-decomposable group is a \emph{hierarchically hyperbolic group}. This theorem can be thought of as a hierarchically hyperbolic group version of two corollaries of the Bestvina--Feighn combination theorem that give necessary and sufficient conditions for a hyperbolic-2-decomposable group to be hyperbolic (\cite[Corollary Section 7]{BestvinaFeighnCombination} and \cite[Corollary 2.3]{BestvinaFeighnCombinationAddendum}). 

Before discussing the theory of hierarchically hyperbolic groups, let us point out some consequences of hierarchical hyperbolicity.  
\begin{corollary}\label{theorem: main theorem applications} Let $G$ be a hyperbolic-2-decomposable group. If $G$ is balanced, then
\begin{enumerate}
\item $G$ has finite asymptotic dimension (\cite[Theorem A]{HHSAsdim2015});
\item $G$ is coarse median in the sense of Bowditch (\cite[Definition 2.1]{bowditch2013coarse}) and, therefore, satisfies a quadratic Dehn function (\cite[Corollary 7.5]{HHSII});
\item every top-dimensional quasi-flat in \(G\) is at bounded distance from a union of standard orthants \cite[Theorem~A]{HHSFlats};
\item if $G$ is virtually torsion-free then either $G$ has uniform exponential growth or there exists a space \(E\) such that \(G\) is quasi-isometric to $\mathbb{Z}\times E$ (\cite[Theorem 1.1]{abbott2019exponentialgrowth});
\item \(G\) satisfies the Morse local-to-global property, in particular, given stable subgroups \(H\), \(K\) of \(G\), there are sufficient conditions to ensure \(\langle H, K \rangle \cong H\ast_{H \cap K}K\) (\cite[Theorem G and Theorem 4.20]{russell2019local});
\item a finite family of subgroups \(\{H_i\}\) is hyperbolically embedded in \(G\) if and only if it is a family of almost-malnormal strongly quasiconvex subgroups (\cite[Theorem 8.1]{russellsprianotran:convexity});
\item there is a linear bound on the length of the shortest conjugator for any pair of conjugate Morse elements in $G$ (\cite[Theorem A]{abbott2018conjugator});
\item \(G\) does not contain distorted cyclic subgroups (\cite[Theorem 7.1]{HHSBoundaries} and \cite[Theorem 3.1]{ durham2018corrigendum}).
\end{enumerate}
\end{corollary}
To the best of our knowledge, all the items of Theorem \ref{theorem: main theorem applications} were previously unknown for hyperbolic-2-decomposable groups.

Hierarchically hyperbolic groups (HHGs) were introduced by   Behrstock, Hagen and Sisto in \cite{HHSI} and \cite{HHSII} and have been widely studied since. The key intuition that led to the definition is that several aspects of the machinery developed by Masur and Minsky to study the mapping class groups in \cite{MasurMinskyI,MasurMinskyII} could be generalized to a much larger class of groups. Examples of hierarchically hyperbolic groups are most (conjecturally all) cubical groups (\cite{HagenSusse}), in particular all right-angled Artin and Coxeter groups; groups hyperbolic relative to peripheral hierarchically hyperbolic groups (\cite{HHSII}); many 3-manifold groups (\cite{HHSII}); certain quotients of mapping class groups (\cite{behrstock2020combinatorial, HHSAsdim2015}) and many more. Given the number of interesting examples in the class of  hierarchically hyperbolic groups, it is perhaps surprising that being a hierarchically hyperbolic group has several deep consequences, as Corollary \ref{theorem: main theorem applications} suggests. As the full definition of hierarchically hyperbolic group is quite technical,  we postpone it until Section~\ref{subsection: basics of hhs}. For the time being, it is enough to know that a hierarchically hyperbolic group structure $(G,\mathfrak{S})$ on \(G\) consists of a collection of $\delta$--hyperbolic spaces 
$\{\mathcal{C}V\mid V\in\mathfrak{S}\}$, and projections $\pi_V$ from $\mathrm{Cay}(G)$ onto the various hyperbolic spaces $\mathcal{C}V$, satisfying a number of axioms.

A property of hierarchically hyperbolic groups that plays a key role in this paper is the fact that they do not have distorted cyclic subgroups. As a first application, this property provides a source of examples of groups that are not hierarchically hyperbolic. The standard example of groups with distorted cyclic subgroups are Baumslag--Solitar groups $BS(m,n)=\langle a,t\mid ta^nt^{-1}=a^m\rangle$ with \(\vert m \vert \neq \vert n \vert\), often called \emph{non-Euclidean} Baumslag--Solitar subgroups. In particular, if a group has a non-Euclidean Baumslag--Solitar subgroup, then the group cannot be hierarchically hyperbolic. When considering hyperbolic-2-decomposable groups, this is a clear  obstruction to hierarchical hyperbolicity. The main result of the paper is that, up to some issues with torsion, this is essentially the only obstruction. More precisely, we say that a \emph{non-Euclidean almost Baumslag--Solitar group} is a group generated by two infinite order elements \(a,b\) such that we have \(ba^{m}b^{-1} = a^{n}\) for some \(\vert m \vert \neq \vert n \vert\).  Equivalently, a non-Euclidean almost Baumslag--Solitar group is a quotient of a non-Euclidean Baumslag--Solitar group where the image of the element \(a\) has infinite order (see Definition~\ref{def: almost bs groups} and successive remarks).

Then we have the following theorem, which is a shortened version of Corollary~\ref{corollary: main result version 2}.

\begin{thm}\label{thmi: HHG iff no BS subgroups}
Let $G$ be a hyperbolic-2-decomposable group. The following are equivalent.\begin{enumerate}
    \item $G$ admits a hierarchically hyperbolic group structure.
    \item $G$ does not contain a distorted infinite cyclic subgroup.
    \item $G$ does not contain a non-Euclidean almost Baumslag--Solitar group.
\end{enumerate}
Moreover, if $G$ is virtually torsion-free, condition (3) can be replaced by 
\begin{enumerate}
    \item[3'.] $G$ does not contain a non-Euclidean Baumslag--Solitar group. 
\end{enumerate}
\end{thm}

We believe that Item (3') of Theorem \ref{thmi: HHG iff no BS subgroups} should be true even without the assumption of $G$ being virtually torsion-free. See the Questions section for further discussion.
 
Theorem \ref{thmi: HHG iff no BS subgroups} is a combination theorem for hierarchically hyperbolic groups: if edge  and vertex groups of a graph of groups satisfy certain conditions, then the (fundamental group of the) graph of groups is hierarchically hyperbolic. Combination theorems for hierarchically hyperbolic groups are not new, and indeed Theorem \ref{thmi: HHG iff no BS subgroups} relies on a combination theorem for hierarchically hyperbolic groups (\cite[Theorem C]{berlai2018refined}). However, there are important differences with many other results of the same form. Firstly, Theorem \ref{thmi: HHG iff no BS subgroups} can be used as a black-box. The statement of the theorem is elementary and does not require familiarity with hierarchically hyperbolic groups to be understood. 
Secondly, Theorem \ref{thmi: HHG iff no BS subgroups} imposes no condition of geometric nature on the edge embeddings, such as, say, the images of the edge groups in the vertex groups to form an \emph{almost-malnormal} collection (Definition \ref{def: almost malnormal collection}). 

Unlike Theorem \ref{thmi: HHG iff no BS subgroups}, several important results in the literature require almost-malnormality of edge groups. For instance,
the groundbreaking work of Haglund--Wise and Hsu--Wise (\cite{HaglundWise:ACombination, HsuWise:Cubulating}), which is crucial in Agol's proof of the virtual Haken conjecture (\cite{AgolHaken}), provides a combination theorem for virtually compact special groups where one of the key condition is the almost-malnormality of the edge groups in the vertex groups. 

\medskip

\textbf{Detecting almost Baumslag--Solitar subgroups:} 
In general, checking whether a given graph of groups contains an almost Baumslag--Solitar subgroup may be challenging. For this reason, we introduce the notion of \emph{balanced edges}. An edge \(e\) of a graph of groups \(\mc{G}\) is a balanced edge if for every infinite order element \(g \in G_e\) and \(h \in \pi_1(\mc{G}-e)\)
\[\text{ if }hg^ih^{-1}=g^j \text{ then }|i|=|j|.\]

We then have the following criterion to detect almost Baumslag--Solitar subgroups.

\begin{thm}\label{thm: non-euclidean BS iff unbalanced edge intro section}
Let \(\mc{G}\) be a graph of groups where none of the vertex groups contain distorted cyclic subgroups. Then \(\pi_1(\mc{G})\) contains a non-Euclidean almost Baumslag--Solitar subgroup if and only if \(\mc{G}\) has an unbalanced edge.
\end{thm}

The proof of this result can be found in Theorem \ref{thm: non-euclidean BS iff unbalanced edge}.

The absence of unbalanced edges is a significantly weaker condition than almost-malnormality of edge groups. Indeed, whenever the underlying graph of \(\mc{G}\) is a tree, all the edges will be automatically balanced (Remark \ref{remark: edges in tree are balanced}), even if the edge groups do not form almost-malnormal collections.  
In particular, we conclude that if \(G = H_1 \ast_C H_2\) where \(H_i\) are hyperbolic and \(C\) is 2-ended, then \(G\) does not contain non-Euclidean almost Baumslag--Solitar subgroups. As a consequence, we have the following. \begin{corollary}\label{corol: corollary to mainT intro version}
Let \(G = H_1 \ast_C H_2\) where \(H_i\) are hyperbolic and \(C\) is 2-ended. Then \(G\) is a hierarchically hyperbolic group. 
\end{corollary}

This result is proved in Corollary \ref{corol: corollary to mainT}.

Let us discuss briefly the main aspects of the proof of Theorem \ref{thmi: HHG iff no BS subgroups}.

\textbf{Idea of the proofs:} Ultimately, our goal is to show that the combination theorem for hierarchically hyperbolic groups (\cite{berlai2018refined}, see \ref{comb_thm_ver2}) can be applied to hyperbolic-2-decomposable groups with no non-Euclidean almost Baumslag--Solitar subgroups. However, verifying the hypotheses of the combination theorem present several challenges. Firstly, we need to construct hierarchically hyperbolic group structures on all vertex and edge groups, and then verify that such structures satisfy the hypotheses of the combination theorem. The key hypothesis to check here is that the embedding of the edge groups into the vertex groups are \emph{glueing hieromorphisms}. Avoiding the technicalities (discussed in  Definition \ref{def: glueing_hieromorphism}), given hierarchically hyperbolic groups $(G,\mathfrak{S}),(G',\mathfrak{S}')$ a necessary condition to the existence of a glueing hieromorphism is that  $\mathfrak{S}\subseteq\mathfrak{S}'$. That is to say, the set of hyperbolic spaces forming the hierarchical structure of \(G\) is a subset of the hierarchical structure of \(G'\). This is one of the key difficulties in equipping the edge and vertex groups with hierarchically hyperbolic structure: the structure of every edge needs to be a subset of the structure of both the vertices it is incident to. However, since we do not require almost-malnormality, the structures of two edge groups incident to the same vertex group may interact with each other, which in turn influences the structures of the other vertices adjacent to such edges and so on.  

To overcome this issue, we find hierarchically hyperbolic structures on the edge groups that are compatible whenever the edge groups are not almost-malnormal. This is done by fixing a reference group, in this case the dihedral group, and pulling back the same hierarchical structure to all the various edge groups. The key notion used is the one of \emph{linearly parametrizable graph of  groups}. A graph of groups \(\mc{G}\) is linearly parametrizable if there is a homomorphism \(\Phi \colon \pi_1(\mc{G}) \to \dihed\) such that \(\Phi\vert_{G_u}\colon G_u \to \dihed\) is a quasi-isometry for each vertex or edge group \(G_u\). Similarly a group \(G\) is linearly parametrizable if \(G = \pi_1(\mc{G})\) for some linearly parametrizable graph of groups \(\mc{G}\). Note that, by definition, a linearly parametrizable group is (2-ended)-2-decomposable. Perhaps the reader will not be surprised that balancedness is the key for the converse to hold. This result is a consequence of Theorem \ref{thm: balanced edges and BS version 2} 
\begin{thm}
Let \(G\) be a group. Then \(G\) is linearly parametrizable if and only if \(G\) is (2-ended)-2-decomposable and balanced. 
\end{thm}
As hinted before, linearly parametrizable graph of groups satisfy the hypotheses of the combination theorem for hierarchically hyperbolic groups, yielding a version of Theorem \ref{thmi: HHG iff no BS subgroups} that holds for (2-ended)-2-decomposable groups.

To  extend such a result to hyperbolic-2-decomposable groups, we introduce the concept of \emph{conjugacy graph} (Definition~\ref{definition: conjugacy graph}). The conjugacy graph is a graph of groups associated to each commensurability class of edge groups. If the group \(\pi_1(\mc{G})\) is balanced, then all the conjugacy graphs are linearly parametrized. Then, using a relative hyperbolicity argument, we construct hierarchically hyperbolic structures on the vertex groups that are compatible with the various conjugacy graphs (Theorem \ref{theorem: readaptation of relative HHG}).

\subsection{Questions}
\textbf{The non virtually torsion-free case:}
our results are stated differently for the case of virtually torsion-free groups. The main problem being that we could not determine in the class of hyperbolic-2-decomposable groups whether all non-Euclidean almost Baumslag--Solitar groups contain a Baumslag--Solitar subgroup.
\begin{question}Does every non-Euclidean almost Baumslag--Solitar subgroup of a hyperbolic-2-decomposable group contain a non-Euclidean Baumslag--Solitar subgroup? 
\end{question}
We stress that this question has a positive answer for certain torsion-free groups. In \cite[Proposition~7.5]{levitt2015quotients} the author shows that the question has a positive answer for GBS groups. In \cite[Proposition 9.6]{button2015balanced} the author extends the result to (torsion-free hyperbolic)-2-decomposable groups. However, the results appearing in those papers rely heavily on the absence of torsion. As we will see in Section \ref{section: graph of infinite virtually cyclic groups}, it is enough to assume that $G$ is virtually torsion-free. It is perhaps also worth noting that a graph of virtually torsion-free groups may not have a virtually torsion-free fundamental group, even when the edge groups are assumed to be of finite index in its neighbouring vertex groups. This is illustrated, for instance, in the examples appearing in these \href{https://mathoverflow.net/questions/330632/is-an-hnn-extension-of-a-virtually-torsion-free-group-virtually-torsion-free}{mathoverflow replies\footnote{https://mathoverflow.net/questions/330632/is-an-hnn-extension-of-a-virtually-torsion-free-group-virtually-torsion-free}.}

\medskip

\textbf{Generalization to HHG-2-decomposable}
In our proofs, hyperbolicity of the edge groups is used only in Theorem~\ref{theorem: readaptation of relative HHG} and Lemma~\ref{remark: elementarizer}. Thus we expect that finding appropriate replacements for the two results above will yield a sufficient condition for a (hierarchically hyperbolic)-2-decomposable group to be hierarchically hyperbolic. However, the question becomes harder when asking for a full characterization. 
As remarked before, all hierarchically hyperbolic groups are balanced, hence balancedness is surely a necessary condition in Question 2.

\begin{question}
Under which conditions a (hierarchically hyperbolic)-2-decomposable group is hierarchically hyperbolic?
\end{question}

A possible strategy to answer this question would be to extend the tools developed in Section~\ref{section: graph of word hyperbolic groups} to the class of hierarchically hyperbolic groups. That is to say, provide conditions guaranteeing that the hierarchically hyperbolic structure of edge groups can be included in the one of the vertex group. 

However, we don't think this strategy would work in the general case.  For instance, consider $\mathbb{Z}^2$-2-decomposable groups (also known as \emph{tubular groups}). If one vertex has three incoming edges, defining pairwise linearly independent lines, there is no straightforward way of defining a hierarchically hyperbolic group structure on $\mathbb{Z}^2$ that contains each edge group.

\subsection*{Acknowledgments} 
The first author would like to thank Montserrat Casals-Ruiz, Mark Hagen and Ilya Kazachkov for numerous and attentive discussions on the work presented in this paper and also for reading and commenting very early versions of it. The second author would like to thank Alessandro Sisto for numerous helpful discussions and suggestions. 
We thank Daniel Woodhouse and Matteo Pintonello for several early inputs on this paper. We would also like to thank Jason Behrstock and Harry Petyt for numerous comments that helped to improve the exposition of the paper.

The first author was supported by the University of the Basque Country through the grant PIF17/241. He also acknowledges the support of the ERC grant PCG-336983, Basque Government Grant IT974-16, and Ministry of Economy, Industry and Competitiveness of the Spanish Government Grant PID2019-107444GA-I00.

The second author was partially supported by the Swiss National Science Foundation (grant \# 182186).

\section{Graph of groups and balanced groups}\label{section: graph of groups and balanced groups}

\subsection{Graph of groups}
We start by recalling the definition of a graph of groups. As usually with graph of groups, we will consider oriented edges. Many of the results of this subsection are probably known to experts. We include them here for the reader's convenience and to uniformize notation.
\begin{defn}
A \emph{graph} \(\Gamma\) consists of sets \(V(\Gamma)\), \(E(\Gamma)\) and maps 
\begin{align*}
    E(\Gamma) &\to V(\Gamma) \times V(\Gamma);  &  E(\Gamma) &\to E(\Gamma) \\
    e &\mapsto (e^{+}, e^{-}) &  e & \mapsto \bar{e}
\end{align*}
satisfying \(\bar{\bar{e}} = e\), \(\bar{e}\neq e\) and  \(\bar{e}^{-} = e^{+}\).
\end{defn}
The elements of \(V(\Gamma)\) are called \emph{vertices},  the ones of \(E(\Gamma)\) are called \emph{edges}, the vertex \(e^{-}\) is the \emph{source} of \(e\), \(e^{+}\) is the \emph{target} and \(\bar{e}\) is the \emph{reverse edge}. A graph \(\Gamma\) is \emph{finite} if both \(V(\Gamma), E(\Gamma)\) are finite sets.
A \emph{subgraph} of \(\Gamma\) is a graph \(\Gamma'\) such that \(V(\Gamma') \subseteq V(\Gamma)\) and \(E(\Gamma') \subseteq E(\Gamma)\). 
Given a graph \(\Gamma\), it is standard to associate to it a \(\Delta\)--complex \(\vert \Gamma \vert\). We say that \(\Gamma\) is \emph{connected} if \(\vert \Gamma \vert\) is. We say that a graph \(\Gamma\) is a \emph{tree} if \(\vert \Gamma \vert\) is simply connected.
We say that a subgraph \(T\) of \(\Gamma\) is a \emph{spanning tree} if \(V(T) = V(\Gamma)\) and \(T\) is a tree. 
\begin{defn}
A \emph{graph of group} \(\mc{G}\) consists of a finite graph \(\Gamma\), a collection of groups \(\{G_v \mid v \in V(\Gamma)\}\), \(\{G_e \mid e \in E(\Gamma)\}\) and injective homorphisms \(\phi_{e^\pm}:G_e \to G_{e^{\pm}}\) such that 
\begin{enumerate}
    \item \(G_{e} = G_{\bar{e}}\);
    \item \(\phi_{e^{+}} = \phi_{\bar{e}^{-}}\).
\end{enumerate}
\end{defn}
 We will often use the notation \(V(\mc{G})\) to denote \(V(\Gamma)\) and similarly for \(E(\mc{G})\). 
\begin{defn}
Let \(\mc{G}= (\Gamma, \{G_v\}, \{G_e\}, \{\phi_{e^{\pm}}\})\) be a graph of groups. We define the group \(F\mc{G}\) as:
\[F\mc{G} = \left(\bigast_{v \in V(\Gamma)}G_v\right) \ast \left(\bigast_{e \in E(\Gamma)} \langle t_{e} \rangle\right).\]
Let \(T\) be a spanning tree of \(\Gamma\). Then the \emph{fundamental group} of \(\mc{G}\) with respect to \(T\), denoted by \(\pi_1(\mc{G},T)\), is the group obtained adding the following relations to \(F\mc{G}\):
\begin{enumerate}
    \item \(t_e = t_{\bar{e}}^{-1}\);
    \item \(t_e = 1\) if \(e \in E(T)\);
    \item \(t_e \phi_{e^{+}}(x)t_e^{-1} = \phi_{e^{-}}(x)\) for all \(x \in G_e\).
\end{enumerate}
\end{defn}

\begin{rmk}\label{rmk: pi_1 of G does not depend on the spannig tree} 
The group \(\pi_1(\mc{G}, T)\) does not depend on the choice of the spanning tree, meaning that for different spanning trees \(T, T'\) there is an isomorphism \(\pi_1(\mc{G},T )\to \pi_1(\mc{G},T')\). For this reason, we will often denote \(\pi_1(\mc{G},T)\) simply by \(\pi_1(\mc{G})\) (see, for instance \cite[Proposition~20]{SerreTrees}). \end{rmk}

Unless otherwise specified, we will represent the elements of \(\pi_1(\mc{G})\) in the alphabet \(\bigcup_{v \in V(\mc{G})} G_v \cup \bigcup_{e \in E(\mc{G})}\langle t_{e}\rangle\). That is, we write each element \(g \in \pi_1(\mc{G})\) as 
\(g = x_0 x_1 \dots x_k\) where either \(x_i \in G_v\) for some \(v\), or \(x_i = t_{e}^{m}\) for some \(e\in E(\mc{G})\). Moreover, we will assume that if \(1 \neq x_i \in G_v\), then \(x_{i+1} \not\in G_v\), and similarly if \(1 \neq x_i \in \langle t_e \rangle\), then \(t_{i+1} \not \in \langle t_e \rangle\). Note that this is not a restrictive assumption as if  \(x_i, x_{i+1} \in G_v\), then we replace them by the element \(x'= x_{i}x_{i+1} \in G_v\), and similarly for \(\langle t_e \rangle\). Finally, we will assume that if \(x_i\) has the form \(t_e^{\epsilon}\), then \(\epsilon \geq 0\). Indeed, otherwise substitute \(t_e^{\epsilon}\) with \(t_{\bar{e}}^{-\epsilon}\).

For many purposes it is convenient to choose a way to write elements of \(\pi_1(\mc{G})\) that takes the geometry of the graph in account. 
\begin{defn}
A word \(w\) is written in \emph{path form} if 
\[w = g_0t_{e_1}^{\epsilon_1}g_1\ldots t^{\epsilon_n}_{e_n}g_n,\] where $\varepsilon_i=\pm 1$ and we require \(g_i \in G_{e_{i+1}^-}\) and \(g_i \in G_{e_i^+}\), whenever defined, and \(g_0, g_n \in G_v\) for some \(v\). As a consequence, \(e_1, \dots, e_n\) form a closed path in \(\Gamma\). We say that the path form is \emph{based} at \(v\).
\end{defn}

\begin{rmk}\label{rmk: writing elements in path form}
Let \(u\) be any word in the alphabet \(\bigcup G_v \cup \{t_{e}\}_{e \in E(\mc{G})}\). It is always possible to replace \(u\) with some \(p\) written in path form such that \(u\) and \(p\) represent the same element of \(\pi_1(\mc{G},T)\). Moreover, the loop of edges associated can be based at any vertex of \(\mc{G}\). Indeed, suppose that the beginning of \(u\) is of the form \(g_0 g_1\), with \(g_0 \in G_v\), \(g_1 \in G_w\). 
Choose a path  \(e_1, \dots, e_m\) in \(T\) between \(v\) and \(w\). This is always possible since \(T\) is a spanning tree. Then replace the beginning of \(u\) with \(g_0 t_{e_1} 1 t_{e_{2}} \dots t_{e_n}g_1\), where \(1\) represents the trivial element. The case where one (or both) of \(g_0, g_1\) were stable letters is analogous. Since we added only stable letters corresponding to edges in the spanning tree, we did not change the group element represented. Proceeding in this way we obtain a word \(p'\) written in path form that represents the same element of \(u\). Suppose that the loop associated to \(p'\) is based at some vertex \(v\), and we want to have a word based at some other vertex \(w\). Again, by considering a path \(e_1, \dots , e_m\) connecting \(v\) and \(w\) in the spanning tree \(T\), we can conjugate \(p'\) by \(t_{e_1} 1 t_{e_2} \dots t_{e_n}\) to obtain the desired word \(p\). 
\end{rmk}

In particular, every element \(g\in \pi_1(\mc{G})\) can be written in path form. 

\begin{rmk}Following the above argument, an element $\omega$ can always be written in path form $\omega=g_0t_{e_1}^{\varepsilon_1}\ldots t_{e_n}^{\varepsilon_n}g_n$ where $\varepsilon_i=1$ for every $i$. We adopt this notation when considering an element written in path form in a graph of groups where the underlying graph has more than one edge. We believe that this renders the proofs of results more readable. However, if we are working with either a simple HNN extension or free product with amalgamation then we allow $\varepsilon$ to be $-1$.
\end{rmk}

\begin{thm}[\textbf{Normal form}]\label{thm: normal form thm}
Let \(\mc{G}\) be a graph of groups and let \(g = g_0 t_{e_1} \dots t_{e_n}g_n\) be written in path form. Then if \(g = 1\) in \(\pi_1(\mc{G})\), there is \(i\) such that \(e_{i} = \bar{e}_{i+1}\) and \(g_i \in \phi_{{e_i}^{+}}(G_{e_i})\).
\end{thm}

\begin{proof}
This is a well known result. For a detailed proof see \cite[Theorem 16.10]{Bogopolski}.
\end{proof}
\begin{rmk}
Note, when the underlying graph of \(\mc{G}\) is a tree, we have \(t_{e_i}= 1\) in \(\pi_1(\mc{G})\) for all edges \(e_i\). Hence, we can greatly simplify path forms to simply be \(g_0 \dots g_n\) where \(g_i \in G_v\) for some vertex \(v\) and \(g_{i \pm 1} \not \in G_v\), and modify the normal form theorem accordingly.
\end{rmk}

\begin{defn}
Let \(\mc{G}\) be a graph of groups. A path word \(g = g_0 t_{e_1} \dots t_{e_n}g_n\) is written in \emph{reduced form} if for each \(i\) such that \(e_i = \bar{e}_{i+1}\) it follows that \(g_i \not\in \phi_{t_{e_i^+}}(G_{e_i})\).
\end{defn}
\begin{corollary}\label{cor: existence reduced form}
For every \(g \in \pi_1(\mc{G},T)\) and \(v \in V(\mc{G})\) it is possible to write \(g\) in a reduced form based at the vertex \(v\).
\end{corollary}
We recall now the notion of commensurable subgroups. Note that this should not be confused with the weaker condition of \emph{abstract commensurability:}  two groups \(A, B\) are abstractly commensurable if they contain isomorphic finite index subgroups.
\begin{defn}\label{def: commensurable subgroups}
 Let \(G\) be a group and \(A, B \leq G\) be subgroups. We say that \(A\) and \(B\) are \emph{commensurable} if there exists \(g \in G\) such that \(gAg^{-1} \cap B\) has finite index in \(B\) and \(A \cap g^{-1} B g\) has finite index in \(A\).

Moreover, we say that two elements $a,b\in G$ are non-commensurable if $\langle a\rangle$ and $\langle b\rangle$ are non-commensurable in $G$.
\end{defn}

\begin{rmk}When \(A, B \leq G\) are infinite virtually cyclic, being commensurable amounts to the existence of some \(g \in G\) such that $\lvert A^g\cap B\rvert = \infty$. Conversely, if $|A^g\cap B|<\infty$ for every $g\in G$ then $A$ and $B$ are not commensurable in $G$.
\end{rmk}

A handy application of the normal form Theorem is the following.
\begin{lemma}\label{lemma: conjugacy path in G}
Let \(\mc{G}\) be a graph of groups, let \(v, w \in V(\mc{G})\) and \(x \in G_v -\{1\}\), \(y \in G_w-\{1\}\). Then \(x,y\) are conjugate in \(\pi_1(\mc{G},T)\) if and only if there is a sequence of edges \(e_1, \dots, e_n\) between \(v\) and \(w\) and elements \(g_i\) satisfying \(g_i \in G_{e_{i}^{+}}, g_i \in G_{e^{-}_{i+1}}\), whenever defined, such that:
\[(g_0 t_{e_1} g_{1} \dots t_{e_n}g_n) x (g_0 t_{e_1} g_{1} \dots t_{e_n}g_n)^{-1} = y.\]
Moreover, for each \(g_i\) we have \( \phi_{e_{i+1}^{-}}(G_{e_{i+1}}) \cap g_i\phi_{e_{i}^{+}}(G_{e_i})g_{i}^{-1} \neq \{1\}\).
\end{lemma}
\begin{proof}
One implication is clear, we need to show the other. 
Suppose \(x,y\) are conjugate and let \(h \in \pi_1(\mc{G})\) be such that \(hxh^{-1}=y\). 
By Corollary \ref{cor: existence reduced form}, there is a reduced path word \(u= u_0 t_{e_1} u_{1} \dots t_{e_m}u_m\) based at the vertex \(v\) that represents \(h\). Choose a shortest path \(f_1, \dots, f_s\) of \(T\) that connects \(w\) and \(v\) and let \(p = t_{f_1} 1 t_{f_2} \dots t_{f_s}\). Then we have \((pu)x(pu)^{-1} = y\), where both sides of the equations are path words based at \(w\). If we multiply by \(y^{-1}\), we have that \((pu)x(pu)^{-1}y^{-1} = 1\), where both sides of the equation are path words. Spelling it out we have:
\[\left[ \left(t_{f_1} 1 t_{f_2} \dots t_{f_s}\right)\left( u_0 t_{e_1} u_{1} \dots t_{e_m}u_m \right)\right]x\left[ \left( u_m^{-1} t^{-1}_{e_m} \dots  u^{-1}_{1}  t^{-1}_{e_a}u^{-1}_0 \right)\left(t^{-1}_{f_1} 1 t^{-1}_{f_2} \dots t^{-1}_{f_s}\right)\right]y^{-1} = 1.\]

 By the normal form Theorem (Theorem~\ref{thm: normal form thm}), in the left hand side of the equation there is a subword of the form \(t_{e} g t_{\bar{e}}\), with \(g \in \phi_{e^{+}}(G_e)\). Our goal is to perform reductions to assume that every such occurrence contains the \(x\). So, suppose this is not the case. Without loss of generality the subword must appear in \(\left[ \left(t_{f_1} 1 t_{f_2} \dots t_{f_s}\right)\left( u_0 t_{e_1} u_{1} \dots t_{e_m}u_m \right)\right]\). Since \(u\) was assumed to be reduced and \(f_1, \dots, f_s\) is a shortest path, the subword must be \(t_{f_s}u_0t_{e_1}\), where \(u_0 = \phi_{f_s^{+}}(z)\) for some \(z \in G_{f_s}\). Then replace \(t_{f_s}u_0t_{e_1}\) by \(\phi_{f^{-}_s}(z)\), and perform the symmetric change on the other side of the \(x\). Note that this process reduces the length of the path \(f_1, \dots, f_s\) by one. In particular, it has to terminate. 
 
 So, assume that no reduction can be performed in \(pu = \left[ \left(t_{f_1} 1 t_{f_2} \dots t_{f_s}\right)\left( u_0 t_{e_1} u_{1} \dots t_{e_m}u_m \right)\right]\). If \(pu = h_0 \in G_w\), and hence \(x,y\in G_w\) are conjugate in \(G_w\), we are done. So suppose this is not the case. We need to have that \(u_m x u_m^{-1} = \phi_{e_m^{+}}(z)\) for some \(z \in \G_{e_m}\). Substitute \(t_{e_m} u_m x u_m^{-1}t_{e_m}^{-1}\) with \(\phi_{e_m^{-}}(z)\).  
 By proceeding as above, we obtain the claim for each \(u_i\).\end{proof}

Whenever we are working with a graph of groups, it is often the case that we are interested in studying a subgraph of groups. To that end, we adopt the following notation.

\begin{notation}\label{notation: subgraph notation}
Let $\mathcal{G}$ be a graph of groups and $\Gamma$ its underlying graph. If $\Lambda\subseteq\Gamma$ is a connected subgraph, then we can define the subgraph of groups $\mathcal{G}|_{\Lambda}$, where the underlying graph is $\Lambda$, every vertex and edge in $\Lambda$ has the same associated groups as in $\mathcal{G}$.

We call $\mathcal{G}|_{\Lambda}$ the \emph{subgraph of groups spanned by $\Lambda$}.
\end{notation}

\begin{lemma}
Let $\mc{G}$ be a graph of groups and let $\Lambda \subseteq \Gamma$ be a subgraph. Let $T'\subseteq\Lambda$ be a spanning tree of $\Lambda$ defined as $T'=T\cap\Lambda$. Then, there exists a group injection $\pi_1(\mc{G}|_{\Lambda},T')\hookrightarrow \pi_1(\mc{G},T)$.
\end{lemma}

\begin{proof}We define the function $\iota:\pi_1(\mc{G}|_{\Lambda},T')\to\pi_1(\mc{G},T)$ on generators such that $\iota|_{G_v}$ is the identity for every $v\in\Lambda$. Let  $g\in\pi_1(\mc{G},T)$ be written in reduced form as $g=g_0t_{e_1}^{\varepsilon_1}\ldots t_{e_n}^{\varepsilon_n}g_n$. By the normal form theorem, $\iota(g)=1$ if and only if there exists $i$ such that $e_i=\overline{e_{i+1}}$ and $g_i\in\phi_{e_i
^+}(G_{e_i})$. By definition, this is equivalent to $g=1$ in $\pi_1(\mc{G},T')$.
\end{proof}

\subsection{Balanced groups}
As mentioned in the introduction, a fundamental notion throughout the paper is the notion of balanced group.

\begin{defn}
Let \(G\) be a group and $g\in G$. We say that \(g\) is balanced either if \(g\) has finite order, or if whenever $g^n$ is conjugate to $g^m$, it must follow  $|n|=|m|$. We say that a group $G$ is \emph{balanced} if every  element is balanced. 
\end{defn}

\begin{lemma}[{\cite[Lemma~4.14]{wise2000subgroup}}]\label{lemma: finite index and balance}
Let $G$ be a group and assume that there exists a balanced subgroup $H$ of $G$ of finite index. Then, $G$ is balanced.
\end{lemma}

We are now going to study how balanced groups behave under amalgamated products and HNN extension over virtually cyclic groups. A key property of virtually cyclic groups that will be used throughout the paper is that if \(a, b\) are infinite order elements of a virtually cyclic group, then there are \(N, M\) such that \(a^N = b^M\).

\begin{lemma}\label{lemma: amalgam and balance}
Let \(C\) be a virtually cyclic group and \(G = A\ast_C B\). Then \(G\) is balanced if and only if \(A, B\) are. 

\end{lemma}
\begin{proof}
One implication is clear. To show the converse, let $g\in G$ be an infinite order element and let $h\in G$ be such that $hg^nh^{-1}=g^m$ for $|n|\neq|m|$. If $g$ is acts hyperbolically on the Bass-Serre tree $T$ corresponding to $G$, then the translation length $\ell_G(g)$ is positive. Moreover, $\ell_G(g^n)=|n|\ell_G(g)$ and $\ell_G(hgh^{-1})=\ell_G(g)$. Thus, if $hg^nh^{-1}=g^m$ then $|n|=|m|$, which is a contradiction. Thus, we can assume that $g$ acts elliptically on $T$. 

Therefore, there exists $x$ such that $xgx^{-1}$ belongs in $A$ or $B$. Assume without loss of generality that  $xgx^{-1}\in A$.  We have \begin{equation}\label{eq: first amalgamated product of balanced} (xhx^{-1})(xgx^{-1})^n(xhx^{-1})^{-1}=(xgx^{-1})^m.\end{equation}
If we write \(a = (xgx^{-1})\in A\) and \(k = xhx^{-1}\), Equation \eqref{eq: first amalgamated product of balanced} becomes \(ka^{n}k^{-1} = a^m\). Write \(k\) in normal form \(k_0 \cdots k_s\), where \(k_i \in A-1\) or \(B-1\). We have 
\[(k_0\cdots k_s)a^{Tn}(k_0 \cdots k_s)^{-1} a^{-Tm} = 1.\]
There are now two cases. First, assume that no powers of \(a\) can be conjugated into \(C\), for instance, this happens whenever \(\vert C \vert \leq \infty\). Then by the normal form theorem, \(s=0\) \(k_0 \in A\) and hence \(A\) was not balanced. 

So suppose that there is some power \(a^{\epsilon}\) of \(a\) that can be conjugated into \(C\). Up to conjugating \(a\) and \(k\) and taking powers of \(a\), we can assume that \(a \in C\) and \(k a^{n} k^{-1} = a^{m}\) holds. Again, consider the normal form  \(k = k_0 \dots k_s\). We will proceed by induction on \(s\). 

\textit{Case \(s = 0\)}. In this case we have \(k_0 a^{n} k_0^{-1} = a^m\). Since \(a \in C\), if \(k_0 \in A\) (resp. \(B\)), we have that \(A\) (resp. \(B\)) is unbalanced.  

\textit{Induction step}. Suppose that the claim holds for \(k\) with normal-form length \(s-1\). We will show that it holds for length \(s\). Consider the equation \(ka^{n} k^{-1} = a^{m}\) and assume that \(k\) has normal-form length \(s\).  Observe that for each \(T\) the equation \(ka^{Tn}k^{-1} = a^{Tm}\) still holds. We will show that, for \(T\) large enough, we can write \(ka^{Tn}k^{-1} = a^{Tm}\) as \(k'c^{n'}(k')^{-1} = c^{m'}\) with \(c \in C\), \(\vert n' \vert \neq \vert m'\vert\) and \(k'\) with normal-form length at most \(s-1\). Then we are done by induction hypothesis. 

We have 
\[(k_0\cdots k_s)a^{n}(k_0 \cdots k_s)^{-1}  = a^{m}.\]
By the normal form theorem,  \( b = k_sa^{n}k_s^{-1} \in C\).  Since \(C\) is 2-ended, there is \(c \in C\) and \(P_1, P_2, P_3, P_4\) such that \(a^{P_1} = c^{P_2}\) and \(b^{P_3} =  c^{P_4}\). Let \(K = k_0\cdots k_{s-1}\). Then we have
\begin{equation}\label{eq: second amalgamated product of balanced}
    K k_sa^{P_1 P_3n}k_s^{-1} K^{-1} = a^{P_1P_3m}
\end{equation}
Let's focus on the left-hand side only, conjugating it by \(K\). We have
\[ k_s c^{P_2P_3 n} k_s^{-1}  = k_s a^{P_1P_3 n} k_s^{-1}= b^{P_1P_3}  = c^{P_1P_4}.\]
Since \(k_s\) belongs to either \(A\) or \(B\), all the elements of the above series of equations are in one between \(A,B\), say \(A\). Since \(A\) is balanced, we need to have \(\vert P_2  P_3 n \vert = \vert P_1 P_4 \vert\). Thus, up to possibly substituting \(n\) with \(-n\), we can write the left-hand-side of Equation \eqref{eq: second amalgamated product of balanced} as \(K c^{P_2 P_3 n} K^{-1}\). Now, applying the equality \(a^{P_1} = c^{P_2}\) to the right-hand-side of Equation \eqref{eq: second amalgamated product of balanced}, we have 
\[Kc^{P_1P_4}K^{-1}=K c^{P_2 P_3 n} K^{-1} = c^{P_2 P_3 m}.\]
We are now done by induction hypothesis.
\end{proof}

By applying repeatedly the previous lemma, we obtain the following corollary.

\begin{corollary}\label{corollary: tree product and balance}
If $G$ is a balanced-2-decomposable group such that the underlying graph is a tree, then $G$ is balanced.
\end{corollary}

It is straightforward to check that HNN extensions of balanced groups are not balanced in general: simply consider $BS(2,3)$ as the HNN extension $\langle a,t\mid ta^2t^{-1}=a^3\rangle\cong\langle a\rangle\ast_{ta^2t^{-1}=a^3}$. 

To finish this subsection we include results that give sufficient conditions for an HNN extension over a balanced group to be balanced. We stress that these results are modified versions of \cite[Proposition 6.3]{button2015balanced} and \cite[Theorem 6.4]{button2015balanced}. They have been modified as to allow torsion.

\begin{prop} \label{prop: balance and absence of conjugation}
Let \(H\) be a balanced group, \(A, B \leq H\) be virtually cyclic subgroups and \(\phi\colon A \to B\) be a isomorphism. Let \(G = H \ast_{\phi}\). Then,
\begin{enumerate}
    \item If $g\in H$ but no power of $g$ is conjugate in $H$ into $A\cup B$ then $g$ is still balanced in $G$.
    \item If $A$ and $B$ are non-commensurable in $H$, then $G$ is also a balanced group. 
\end{enumerate}
\end{prop}

\begin{proof}
Suppose \(g\) was not balanced in \(G\). Hence there is \(h\in G-H\) such that \(hg^{p} h^{-1} = g^{q}\) for some \(\vert p \vert \neq \vert q \vert\). Since \(h \in G-H\), we can write \( h = h_1t^{\varepsilon_1}\ldots h_{r-1}t^{\varepsilon_r}h_r\) in reduced form. By assumption \(h_r g h_r^{-1}\) does not belong to \(A\) nor \(B\), and hence \(h g^{q} h^{-1}\) cannot represent an element of \(H\). Thus, \(h \in H\) and since \(H\) is balanced \(\vert q \vert = \vert p \vert\).

For the second item, we only need to check the balancedeness of elliptic elements in $G$, since a translation length argument similar to that of Lemma \ref{lemma: amalgam and balance} rules out unbalancedeness of hyperbolic elements. Thus, if $G$ is unbalanced, by the first item there must exist an unbalanced infinite order element $h\in H$ such that some power of $h$ can be conjugated into $A\cup B$. Therefore, we can assume without loss of generality that $h\in A\cup B$. Assume that $h=a\in A$. Since $a$ is unbalanced, there is some $g\in G$ such that $ga^ig^{-1}=a^j$ with $|i|\neq |j|$. Let $g=h_1t^{\varepsilon_1}\ldots h_rt^{\varepsilon_r}$ be the reduced form expression in $G$. Since $gh^ig^{-1}=h^j$ has normal form length $1$, there must exist some possible reduction in $(h_1t^{\varepsilon_1}\ldots h_rt^{\varepsilon_r})h^i(h_1t^{\varepsilon_1}\ldots h_rt^{\varepsilon_r})^{-1}$. There are two possible ways that this could happen: either $\varepsilon_r=1$ and $h_rh^ih_r^{-1}\in A$ or $\varepsilon_r=-1$ and  $h_rh^ih_r^{-1}\in B$. If the latter occurs, then the proof is complete, as $h_rh^ih_r^{-1}$ is an infinite order element in $A^{h_r}\cap B$. Assume now that the former case occurs. Since $A$ is a 2-ended balanced group, there must exist $k$ such that $h_ra^{ik}h_r^{-1}=a^{\pm ik}$. Therefore, $t^{\varepsilon_r}h_ra^{ik}h_r^{-1}t^{-\varepsilon_r}=ta^{\pm ik}t^{-1}=b^{\pm ik}$. Again, as before, we have two possibilities: either $h_{r-1}b^{\pm ik}h_{r-1}^{-1}$ belongs in $B$ and $\varepsilon_{r-1}=-1$ or $h_{r-1}b^{\pm ik}h_{r-1}^{-1}$ belongs in $A$ and $\varepsilon_{r-1}=1$. If the latter occurs, the proof is complete. If the former occurs, since $B$ is a 2-ended balanced group, then $h_{r-1}b^{\pm ikk'}h_{r-1}^{-1}=b^{\pm ikk'}$ for some $k'$. We can continue performing reductions in the expression of $ga^ig^{-1}$ and at each step we have the same dichotomy where either the proof is complete or we can continue reducing. Note that at some point of the reduction we obtain $h_i$ such that $A
^{h_i}\cap B$ or $A\cap B^{h_i}$ is infinite. Indeed, otherwise for some $K\neq 0$ the equality $ga^{Ki}g^{-1}=a^{Kj}$ would hold for $|Ki|=|Kj|$, contradicting the assumption. \end{proof}

\begin{corollary}\label{corollary: HNN and balance}
Let $G$ be an HNN extension of the balanced group $H$ with stable letter $t$ and 2-ended associated subgroups $A$ and $B$ of $H$. Let $a\in A,b\in B$ be infinite order elements such that $tat^{-1}=b$. Moreover, suppose that there is $h\in H$ conjugating a power of $a$ to a power of $b$, so that $ha^ih^{-1}=b^j$. Then $G$ is balanced if and only if for every pair of elements $a,b$ as above we have $|i|=|j|$.
\end{corollary}

\begin{proof}
One implication is clear, we now show that $G$ is balanced provided that for every $h\in H$ such that $ha^ih^{-1}=b^j$ for some $i,j$ it follows that $|i|=|j|$.

Assume that $G$ is an unbalanced group. Therefore, by the second assertion in the previous proposition, there must exist some $h'\in H$ such that $A\cap h'Bh'^{-1}$ is infinite. Since HNN extensions are defined up to conjugation of the corresponding embedding maps, by conjugating by $h'$ we can assume that $A\cap B$ is infinite in $H$. By the first assertion in the previous proposition, the only elements that can be unbalanced are those $h\in H$ that can be conjugate in $H$ into $A\cup B$. Thus, we can assume without loss of generality that the unbalanced elements in $G$ belong in $A\cup B$. Therefore, if $G$ is unbalanced, we can assume that for some $a\in A$ there is some $g\in G$ such that $ga^ng^{-1}=a^m$ for some $|n|\neq|m|$. 
We will induct on the length of the reduced form of $g$ to show that  $ga^ng^{-1}=a^m$ implies $|n|=|m|$, obtaining a contradiction.

Let $g=h_0t^{\varepsilon_1}h_1\ldots t^{\varepsilon_r}h_r$ be the reduced expression of $g$. Let us say that $r$ denotes the reduced form length of $g$. Assume that $r=0$. That is to say, $g\in H$. Since $H$ is balanced, we have \(\vert n \vert  = \vert m \vert\).

Assume now that the claim holds for elements of reduced form length \(r-1\), and let  $g$ of reduced form length \(r\) be such that $ga^ng^{-1}=a^m$. We denote by $b\in B$ the element such that $tat^{-1}=b$. Note that if the equation $ga^ng^{-1}=a^m$ holds in $G$, then for every $T$ we have that $ga^{Tn}g^{-1}=a^{Tm}$ for every $T>0$.  Since the element $ga^ng^{-1}=a^m$ belongs in $H$, by the normal form theorem, $ga^ng^{-1}$ must admit some reduction in its reduced form. There are two ways that this reduction can occur: either $\varepsilon_r=1$ and $h_ra^nh_r^{-1}$ belongs in $A$ or $\varepsilon_1=-1$ and $h_ra
^nh_r^{-1}$ belongs in $B$. Note that in the former case, since $A$ is 2-ended and balanced, there must exist some $k$ such that $h_ra^{kn}h_r^{-1}=a^{\pm kn}$. Therefore, $t_r^{\varepsilon}h_ra^{kn}h_r^{-1}t_r^{-\varepsilon_r}=b^{\pm kn}$. In the latter case we have that $h_ra^nh_r^{-1}=b'\in B$. Since $B$ is a 2-ended group, there must exist $l_1,l_2$ such that $(b')^{l_1}=b^{l_2}$. Thus, $h_ra^{nl_1}h_r^{-1}=(b')^{l_1}=b^{l_2}$. By assumption, we must have that $|nl_1|=|l_2|$. Therefore, in the latter case we have that $t^{-1}h_ra^{nl_1}h_r^{-1}t=t^{-1}b^{\pm l_2}t=a^{\pm l_2}=a^{\pm nl_1}$. In both cases, we use the induction step to conclude $|kn|=|km|$ or $|l_1n|=|l_1m|$ respectively. In particular, since \(k \neq 0 \neq l_1\), we conclude \(\vert n \vert = \vert m \vert\).
\end{proof}

\section{Hierarchically hyperbolic groups}\label{subsection: basics of hhs}

\begin{defn}\label{HHS_definition}
A $q$--quasigeodesic metric space $(\mathcal{X},d_\mathcal{X})$ is \emph{hierarchically hyperbolic} if there exist $\delta\geqslant 0$, an index set $\mathfrak{S}$, 
and a set $\{\mathcal{C}W\mid W\in\mathfrak{S}\}$ of $\delta$--hyperbolic spaces $(\mathcal{C}U,d_U)$, such that the following conditions are satisfied:
\begin{enumerate}
\item {\bf (Projections)}\label{axiom1} There is a set $\{\pi_W\colon\mathcal{X}\to 2^{\mathcal{C}W}\mid W\in\mathfrak{S}\}$ of projections that send points in $\mathcal{X}$ 
to sets of diameter bounded by some $\xi\geqslant 0$ in the hyperbolic spaces $\mathcal{C}W\in\mathfrak{S}$. Moreover, there exists $K$ so that all $W\in\mathfrak{S}$, the coarse map $\pi_{W}$ is $(K,K)$--coarsely lipschitz and $\pi_W(\mathcal{X})$\footnote{If $A\subseteq \mathcal{X}$, by $\pi_U(A)$ we mean $\bigcup_{a\in A}\pi_U(a)$.} is $K$--quasiconvex in $\mathcal{C}W$.
\item {\bf (Nesting)} The index set $\mathfrak{S}$ is equipped with a partial order $\sqsubseteq$ called \emph{nesting}, and either $\mathfrak{S}$ is empty or it contains a unique $\sqsubseteq$--maximal element. When $V\sqsubseteq W$, $V$ is nested into $W$. For each $W\in\mathfrak{S}$, $W\sqsubseteq W$, and with $\mathfrak{S}_W$
we denote the set of all $V\in\mathfrak{S}$ that are nested in $W$.
For all $V,W\in\mathfrak{S}$ such that $V\sqsubsetneq W$ there is a subset $\rho_W^V\subseteq \mathcal{C}W$ with diameter at most $\xi$, and a map 
$\rho_V^W\colon \mathcal{C}W\to 2^{\mathcal{C}V}$.
\item\label{A3} {\bf (Orthogonality)} The set $\mathfrak{S}$ has a symmetric and antireflexive relation $\perp$ called \emph{orthogonality}. Whenever $V\sqsubseteq W$ and 
$W\perp U$, then $V\perp U$ as well. For each $Z\in \mathfrak{S}$ and each $U\in\mathfrak{S}_Z$ for which $\{V\in\mathfrak{S}_Z\mid V\perp U\}\neq\emptyset$, 
there exists $\cont_\perp^ZU\in \mathfrak{S}_Z\setminus\{Z\}$ such that whenever $V\perp U$ and $V\sqsubseteq Z$, then $V\sqsubseteq \cont_\perp^ZU$.
\item {\bf (Transversality and Consistency)}\label{HHS_definition_4}
If $V,W\in\mathfrak{S}$ are not orthogonal and neither is nested into the other, then they are transverse: 
$V\pitchfork W$. There exists $\kappa_0\geqslant 0$ such that if $V\pitchfork W$, then there are sets $\rho_W^V\subseteq \mathcal{C}W$ and 
$\rho_V^W\subseteq \mathcal{C}V$, each of diameter at most $\xi$, satisfying
\begin{equation*}\label{consistent_transversal}
\min\bigl\{d_W(\pi_W(x),\rho_W^V),d_V(\pi_V(x),\rho_V^W)\bigr\}\leqslant \kappa_0,\qquad \forall\ x\in\mathcal{X}.
\end{equation*}
Moreover, for $V\sqsubseteq W$ and for all $x\in\mathcal{X}$ we have that
\begin{equation*}\label{consistent_nesting}
\min\bigl\{d_W(\pi_W(x),\rho_W^V),\diam_{\mathcal{C}V}(\pi_V(x)\cup\rho_V^W(\pi_W(x)))\bigr\}\leqslant \kappa_0.
\end{equation*}
In the case of $V\sqsubseteq W$, we have that $d_U(\rho^V_U,\rho^W_U)\leqslant \kappa_0$ whenever $U\in\mathfrak{S}$ is such that either 
$W\sqsubsetneq U$, or $W\pitchfork U$ and $U\not\perp V$.
\item {\bf (Finite complexity)} There is a natural number $n\geqslant0$, the complexity of $\mathcal{X}$ with respect to $\mathfrak{S}$, such that any set of 
pairwise $\sqsubseteq$--comparable elements of $\mathfrak{S}$ has cardinality at most $n$.
\item {\bf (Large links)} There exist $\lambda\geqslant 1$ and $E\geqslant \max\{\xi,\kappa_0\}$ such that, given any $W\in\mathfrak{S}$ and $x,x'\in\mathcal{X}$, 
there exists $\{T_i\}_{i=1,\dots,\lfloor N\rfloor}\subset \mathfrak{S}_W\setminus\{W\}$ such that for all $T\in\mathfrak{S}_W\setminus\{W\}$ either $T\in\mathfrak{S}_{T_i}$ for some $i$, or $d_T(\pi_T(x),\pi_T(x'))< E$, where $N=\lambda d_W(\pi_W(x),\pi_W(x'))+\lambda$. 
Moreover, $d_W(\pi_W(x),\rho_W^{T_i})\leqslant N$ for all $i$.
\item {\bf (Bounded geodesic image)} For all $W\in\mathfrak{S}$, all $V\in\mathfrak{S}_W \setminus \{W\}$ and all geodesics $\gamma$ of $\mathcal{C}W$, either $\diam_{\mathcal{C}V}(\rho_V^W(\gamma))\leqslant E$ or $\gamma\cap \mathcal{N}_E(\rho_W^V)\neq \emptyset$.
\item {\bf (Partial realization)} There is a constant $\alpha$ satisfying: let $\{V_j\}$ be a family of pairwise orthogonal elements of $\mathfrak{S}$, ad let
$p_j\in\pi_{V_j}(\mathcal{X})\subseteq \mathcal{C}V_j$. Then there exists $x\in\mathcal{X}$ such that
\begin{itemize}
\item $d_{V_j}\bigl(\pi_{V_j}(x),p_j\bigr)\leqslant \alpha$ for all $j$;
\item for all $j$ and all $V\in\mathfrak S$ such that $V\pitchfork V_j$ or $V_j\sqsubseteq V$ we have $d_V(\pi_V(x),\rho^{V_j}_V)\leq\alpha$.
\end{itemize} 
\item {\bf (Uniqueness)} For each $\kappa\geqslant 0$ there exists $\theta_u=\theta_u(\kappa)$ such that if $x,y\in\mathcal{X}$ and $d(x,y)\geqslant \theta_u$, then there exists $V\in\mathfrak{S}$ such that $d_V(x,y)\geqslant\kappa$.
\end{enumerate}
The inequalities of the fourth axiom are called \emph{consistency inequalities}.
\end{defn}

\begin{rmk}\label{rmk: HHS are QI invariant}
Being a hierarchically hyperbolic space is a quasi-isometric invariant property. That is, if $(\mathcal{X},\mathfrak{S})$  hierarchically hyperbolic and \(q \colon \mc{X} \to \mc{Y}\) is a quasi-isometry, then $\mathcal{Y}$ is a hierarchically hyperbolic space, and the hierarchical structure coincides with the one of \(\mc{X}\). Indeed, if \(\bar{q}\) is a quasi-inverse of \(q\), define all the hyperbolic spaces as the one of \(\mc{X}\) and the projections as \(\pi_U\circ \bar{q}\). Then checking the axioms is a tedious (but straightforward) work. We will see in the following sections that this is not the case for hierarchically hyperbolic groups (Definition~\ref{def:hhg}). The issue boils down to the fact that the quasi-inverse \(\bar{q}\) might not be equivariant. 
\end{rmk}

\begin{defn}[\bf Hieromorphism]\label{def:HHS_hieromorphism}
Let $(\mathcal{X},\mathfrak{S})$ and $(\mathcal{X}',\mathfrak{S}')$ be hierarchically hyperbolic spaces. A \emph{hieromorphism} is a triple 
$\phi=\bigl(\phi,\phi^\diamondsuit,\{\phi^\ast_U\}_{U\in\mathfrak{S}}\bigr)$, where $\phi\colon \mathcal{X}\to\mathcal{X}'$ is a map, 
$\phi^\diamondsuit\colon \mathfrak{S}\to\mathfrak{S}'$ is an injective map that preserves nesting, transversality and orthogonality,
and, for every $U\in\mathfrak{S}$, the maps $\phi^*_U\colon \mathcal{C}U\to\mathcal{C}\phi^\diamondsuit(U)$ are quasi-isometric embeddings with uniform constants. 

Moreover, the following two diagrams coarsely commute (again with uniform constants), for all $U,V\in\mathfrak{S}$ such that $U\sqsubseteq V$ or $U\pitchfork V$:
\begin{equation}\label{coarsely.commuting.diagrams}
\xymatrix{
\mathcal{X}\ar[r]^{\phi}\ar[d]_{\pi_U} & \mathcal{X}'\ar[d]^{\pi_{\phi^\diamondsuit(U)}}\\
\mathcal{C}U\ar[r]_{\phi^*_U}&\mathcal{C}\phi^\diamondsuit(U)
}\qquad\qquad\qquad 
\xymatrix{
\mathcal{C}U\ar[rr]^{\phi^*_U}\ar[d]_{\rho^U_V} && \mathcal{C}\phi^\diamondsuit(U)\ar[d]^{\rho^{\phi^\diamondsuit(U)}_{\phi^\diamondsuit(V)}}\\
\mathcal{C}V\ar[rr]_{\phi^*_V}&&\mathcal{C}\phi^\diamondsuit(V)
}
\end{equation}
\end{defn}

\begin{defn}[\bf Full hieromorphism]\label{def:fullness_definition}
A hieromorphism $\phi\colon(\mathcal{X},\mathfrak{S})\to(\mathcal{X}',\mathfrak{S}')$ is \emph{full} if: 
\begin{enumerate} 
\item there exists $\xi$ such that the maps $\phi^*_U\colon\mathcal{C}U\to\mathcal{C}\phi^\diamondsuit(U)$ are $(\xi,\xi)$--quasi-isometries, for all $U\in\mathfrak{S}$;
\item if $S$ denotes the $\sqsubseteq$--maximal element of $\mathfrak{S}$, then for all $U'\in\mathfrak{S}'$ nested into $\phi^\diamondsuit(S)$ there exists 
$U\in\mathfrak{S}$ such that $U'=\phi^\diamondsuit(U)$.
\end{enumerate}
\end{defn}

\begin{defn}[\bf{Hierarchically hyperbolic group}]\label{def:hhg}
We say that a group $G$ is \emph{hierarchically hyperbolic} if it acts on a hierarchically hyperbolic space $(\mathcal{X},\mathfrak{S})$  satisfying the following conditions:
\begin{enumerate}
\item The action of $G$ on $\mathcal{X}$ is proper and cobounded;
\item \(G\) acts on \(\mc{X}\) by uniform hieromorphisms (i.e the constants involved in Defintion \ref{def:HHS_hieromorphism} are uniform for every $g\in G$), and the action on \(\mf{S}\) has finitely many orbits.
\end{enumerate}
\end{defn}

\begin{rmk}
By definition, if \((G, \mf{S})\) is a hierarchically hyperbolic group and \(g \in G\), multiplication by \(g\) coarsely satisfies the two diagrams of Equation \eqref{coarsely.commuting.diagrams}. However, it is always possible to modify the structure to obtain commutativity on the nose, as described in \cite[Section 2.1]{durham2018corrigendum}. When considering hierarchically hyperbolic groups, we will always assume such equivariance on the nose. 
\end{rmk}

\subsection{Convexity}
In this paper, we will make use of two notions of convexity. The first one, called hierarchical quasiconvexity, heavily relies on the hierarchical structure. For instance, it is not quasi-isometric invariant. For a more precise account, we refer to \cite{russellsprianotran:convexity}.

\begin{defn}[\bf Hierarchical quasiconvexity]
A subset $Y$ of an HHS $(X,\mf{S})$ is hierarchically quasiconvex if there is a function $k:[0,\infty)\to\mathbb{R}$ such that every $\pi_U(Y)$ is $k(0)$--quasiconvex, and any point $x\in X$ with $d_U(x,Y)\leq r$ for all $U\in\mf{S}$ satisfies $d_X(x,Y)\leq k(r)$.
\end{defn}
Although we will not use this fact, we recall that one of the main motivations of hierarchical quasiconvexity is that hierarchically quasiconvex subsets of an HHS are HHSs themselves \cite[Proposition~5.6]{HHSII}. 

To detect hierarchical quasiconvexity sometimes it is convenient to check a stronger property. 

\begin{defn}[\bf Strong quasiconvexity]
A subset $Y$ of a quasigeodesic space $X$ is said to be \emph{strongly quasiconvex} if there is a function $M:[1,\infty)\to \mathbb{R}$ such that every $\lambda$--quasigeodesic in $X$ with endpoints in $Y$ stays $M(\lambda)$--close to $Y$.
\end{defn}

\begin{thm}[{\cite[Theorem~6.3]{russellsprianotran:convexity}}]\label{thm: RST strong quasiconvex}
Let \((G, \mf{S})\) be a hierarchically hyperbolic group and \(Y\subseteq G\) be a subset. Then if \(Y\) is strongly quasiconvex, it is hierarchically quasiconvex, where the constants determine each other. 
\end{thm}

 A special case of strongly quasiconvex subsets is given by peripheral subgroups of relatively hyperbolic groups. 
\begin{lemma}[{{\cite[Lemma~4.15]{DrutuSapir}}}]\label{lem: peripherals are strongly QC}
Let \(P\) be a peripheral subgroup in the relatively hyperbolic group \(G\). Then \(P\) is strongly quasiconvex. 
\end{lemma}

In the case of hyperbolic spaces, relative hyperbolicity and strong quasi-convexity are intimately related.

\begin{defn}\label{def: almost malnormal collection}
We say that a collection of subgroups $\{H_i\}_{i=1}^n$ of $G$ is almost-malnormal if $H_i\cap gH_jg^{-1}$ is finite unless $i=j$ and $g\in H_i$.
\end{defn}

\begin{thm}[{\cite[Theorem 7.11]{BowditchRelHypGroups}}]\label{thm:bowditch}
Let \(G\) be a hyperbolic group and \(\{H_i\}_{i=1}^n\) be a finite family of subgroups of \(G\). Then $G$ is hyperbolic relative to $\{H_i\}$ if and only if  $\{H_i\}$ is an almost-malnormal family of strongly quasiconvex subgroups.
\end{thm}

To finish we recall one last useful property of peripheral subgroups. 
\begin{lemma}[{\cite[Lemma~4.11]{DrutuSapir}}]\label{lem: cpproj is C. Lipschitz}
Let \(P\) be a peripheral subgroup of the relatively hyperbolic group \(G\). Then the closest point projection on \(P\) is coarsely Lipschitz.\end{lemma}

\subsection{Combination theorem}\label{subsec: combination theorems}

\begin{defn}[\bf Glueing hieromorphism]\label{def: glueing_hieromorphism}
Let \((H, \mf{S}_1)\) and \((G, \mf{S}_2)\) be hierarchically hyperbolic groups. A \emph{glueing hieromorphism} between \(H\) and \(G\) is a group homomorphism  \(\phi \colon H \to G\) that can be realized as a full hieromorphism \((\phi, \phi^\diamondsuit, \phi^\ast_U)\) such that the image \(\phi(H)\) is hierarchically quasi-convex in \(G\) and the maps \(\phi^\ast_U \colon \fontact U \to \fontact \phi^\diamondsuit U\) are isometries for each \(U \in \mf{S}_1\). If the map \(\phi \colon H \to G\) is injective, we say that the glueing hieromorphism is injective. 
\end{defn}

\begin{defn}\label{intersectionproperty_definition}
A hierarchically hyperbolic space $(\mathcal{X},\mathfrak{S})$ has the \emph{intersection property} if 
the index set admits an operation $\wedge\colon (\mathfrak{S}\cup\{\emptyset\})\times(\mathfrak{S}\cup\{\emptyset\})\to\mathfrak{S}\cup\{\emptyset\}$ satisfying
the following properties for all $U,V,W\in\mathfrak{S}$:
\begin{enumerate}
\item $V\wedge \emptyset=\emptyset \wedge V=\emptyset$;
\item $U\wedge V=V\wedge U$;
\item $(U\wedge V)\wedge W=U\wedge (V\wedge W)$;
\item $U\wedge V\sqsubseteq U$ and $U\wedge V\sqsubseteq V$ whenever $U\wedge V\in\mathfrak{S}$;
\item if $W\sqsubseteq U$ and $W\sqsubseteq V$, then $W\sqsubseteq U\wedge V$.
\end{enumerate}
\end{defn}
 
\begin{defn}
A hierarchically hyperbolic space $(\mathcal{X},\mathfrak{S})$ is said to have clean containers if $U\perp \cont_\perp^Z U$ for all $U,Z\in\mathfrak{S}$, as originally defined in \cite[Definition 3.4]{ABD}.
\end{defn} 

We also recall a combination theorem for hierarchically hyperbolic groups. 

\begin{thm}[{\cite[Theorem C]{berlai2018refined}}]\label{comb_thm_ver2}
Let $\mc{G} = \bigl(\Gamma,\{G_v\}_{v\in V},\{G_e\}_{e\in E},\{\phi_{e^\pm}\colon G_e\to G_{e^\pm}\}_{ e\in E}\bigr)$ be a finite graph of hierarchically
hyperbolic groups. Suppose that:
\begin{enumerate}
\item each \(\phi_{e^\pm}\) is a $(K,K)$--coarsely Lipschitz map and a glueing hieromorphism;
\item each vertex group has the intersection property and clean containers.
\end{enumerate}
Then \(\pi_1(\mc{G})\) is a hierarchically hyperbolic group with the intersection property and clean containers.
\end{thm}

Not all fundamental groups of a graph of groups have a hierarchical hyperbolic structure, as the following remark shows.

\begin{rmk}\label{remark: absence of BS in HHG}
If $G$ is a hierarchically hyperbolic group, then $G$ cannot have a subgroup isomorphic to $BS(n,m)=\langle a,t\mid ta^nt^{-1}=a^m\rangle$, with \(\vert n \vert \neq \vert m \vert\). Indeed, suppose there is an embedding $\iota: BS(n,m)\hookrightarrow G$. We have that \(\iota(a)\) is an infinite  order element of \(G\). By \cite[Theorem 7.1]{HHSBoundaries} and  \cite[Theorem 3.1]{ durham2018corrigendum}, \(\iota(a)\) is undistorted, which is a contradiction.
\end{rmk}

More generally, if a group $G$ has a hierarchical hyperbolic structure, then it cannot be unbalanced, as it cannot contain infinite undistorted cyclic subgroups. The rest of the paper is dedicated to investigate if the converse also holds for (fundamental group of) graph of groups. More precisely, we show that the converse holds for the class of hyperbolic-2-decomposable groups. 

\section{Hierarchical hyperbolicity of (2-ended)-2-decomposable groups}\label{section: graph of infinite virtually cyclic groups}

In this section, we focus on (2-ended)-2-decomposable groups. That is to say, graphs of groups where every vertex and edge group is 2-ended. We begin the section by recalling some useful results on 2-ended groups.

\subsection{Two-ended groups}

In this subsection, we recall basic results and remarks on the structure of two-ended groups. An important result of these type of groups is known as the structure theorem for infinite virtually cyclic groups. Throughout the paper, we will make use of this fact on many occasions.

\begin{lemma}[{\cite[Lemma 4.1]{wall1967poincare}}]\label{lemma:basic_fact_virt_cyclic}
If $G$ is an infinite virtually cyclic group, then either
\begin{enumerate}
    \item $G$ admits a surjection with finite kernel onto the infinite cyclic group $\mathbb{Z}$, or
    \item $G$ admits a surjection with finite kernel onto the infinite dihedral group $\mathbb{D}_{\infty}$
\end{enumerate}
\end{lemma}

We recall that the \emph{infinite dihedral group} is the group defined by the presentation \(\mathbb{D}_\infty = \langle r, s \mid srs = r^{-1}, s^2 \rangle\). Note that every element of \(\mathbb{D}_\infty\) can be written as \(s^\epsilon r^k\), for \(\epsilon\in \{0,1\}\) and \(k \in \mathbb{Z}\). Moreover, every element of the form \(sr^k\) has order 2, and an element of the form \(r^k\) has infinite order precisely when \(k \neq 0\). Using those observations, we have the following lemma.

\begin{lemma}\label{lem: Kernel to dihedral is always the same}
Let \(G\) be a virtually cyclic group. Let \(\Phi_1\), \(\Phi_2 \colon G \to \mathbb{D}_\infty\) be homomorphisms with finite kernel and finite index image. Then \(\Ker (\Phi_1) = \Ker (\Phi_2)\). 
\end{lemma}
\begin{proof}
As before, \(\mathbb{D}_\infty = \langle a, b \mid bab = a^{-1}, b^2 \rangle\).
Suppose that there is \(g \in G\) such that \(g \in \Ker(\Phi_1)\) and \(g \not \in \Ker(\Phi_2)\). Since \(g \in \Ker(\Phi_1)\), we conclude that \(g\) has finite order, otherwise \(\lvert \Ker(\Phi_1)\rvert = \infty\). Since \(\Phi_2(G)\) has finite index in \(\mathbb{D}_\infty\) there exists \(c \in G\) such that \(\Phi_2(c)\) has infinite order. In particular there exist \(k_1 \in \mathbb{Z}, k_2 \in \mathbb{Z}-\{0\}\) such that \(\Phi_2(g) = ba^{k_1}\) and \(\Phi_2(c) = a^{k_2}\), and so \(\Phi_2(gc) = ba^{k_1+k_2}\). Again, \(gc\) has to have finite order to not contradict \(\lvert \Ker(\Phi_2)\rvert < \infty\) . However, since \(g \in \Ker(\Phi_1)\) we have that \(\Phi_1(gc) = \Phi_1(c)\), and so \(gc\) cannot have finite order. From this we conclude \(\Ker(\Phi_1) \subseteq \Ker(\Phi_2)\). The symmetric argument yields the claim.
\end{proof}

\begin{rmk}\label{rmk: virtually cyclic surjection dichotomy}
Note that an infinite virtually cyclic group $G$ cannot surject onto both $\mathbb{Z}$ and $\mathbb{D}_{\infty}$ with finite kernel. Indeed, assume that two surjective homomorphisms $\Phi:G\to\mathbb{Z}$ and $\Phi':G\to\mathbb{D}_{\infty}$ exist. Since $\mathbb{Z}$ embeds into $\mathbb{D}_{\infty}$ with finite index image, we can regard $\Phi$ as a homomorphism from $G$ to $\mathbb{D}_{\infty}$ with finite kernel and finite index image. Let $s\in\mathbb{D}_{\infty}$ be the generator of order two and let $g\in G$ be an element such that $\Phi'(g)=s$. Since $s^2=1$, we have that $g^2\in\Ker(\Phi')$; by Lemma \ref{lem: Kernel to dihedral is always the same} we have that $g^2\in\Ker(\Phi)$. Since $\mathbb{Z}$ is torsion-free, $\Phi(g)^2=1$ if and only if $\Phi(g)=1$. Since $\Ker(\Phi)=\Ker(\Phi')$, it follows that $\Phi'(g)=1$, which is a contradiction.
\end{rmk}

\subsection{Pulling back hierarchical structures}

Recall that GBS groups are (infinite cyclic)-2-decomposable groups.

\begin{defn}\label{definition: GBS groups}
We say that a group $G$ is a Generalized Baumslag--Solitar group if there exists a finite graph of infinite cyclic groups $\mc{G}$ for which $G\cong\pi_1(\mc{G})$.
\end{defn}

\begin{lemma}\label{lemma: subgroups of graph of 2 ended groups}
Let $G$ be a (2-ended)-2-decomposable group and let $H\leq G$. If $H$ is torsion-free, then $H$ is either a GBS group or a free group.
\end{lemma}

\begin{proof}
Let $G_v$ be a vertex group in $\mathcal{G}$. Since $H$ is torsion-free, there are two possibilities: either $H\cap G_v$ is trivial or it is infinite cyclic. Since every edge group has finite index in its neighbouring vertex groups, if $H\cap G_v$ is trivial, then $H\cap G_w$ is trivial for every other vertex $w$. Then $H$ acts on the Bass-Serre tree corresponding to $\mathcal{G}$ with trivial stabilizers. This is equivalent to $H$ being a free group.

If $H\cap G_v$ is non trivial, then it is of finite index in $G_v$, since $G_v$ is two-ended. Therefore, since the Bass-Serre tree of $\mathcal{G}$ is locally finite, the group $H$ acts with infinite cyclic stabilizers on a locally finite tree. That is to say, $H$ splits as a finite graph of groups with infinite cyclic vertex groups and the result follows.
\end{proof}

\begin{defn}
Let $G,H$ be finitely generated groups and let $S_G,S_H$ be generating sets of $G$ and $H$ respectively. We say that a group homomorphism $f:H\to G$ is a \emph{quasi-isometric homomorphism} if $f:(G,d_{S_G})\to (H,d_{S_H})$ is a quasi-isometry.
\end{defn}

\begin{rmk}
Recall that a group homomorphism $f:G\to H$ yields a quasi-isometry for some (hence, any) generating sets $S_H,S_G$ if and only if $|\Ker(f)|<\infty$ and $|H:\Im(f)|<\infty$.
\end{rmk}

As we have seen in Remark \ref{rmk: HHS are QI invariant}, the hierarchically hyperbolic structure on geodesic metric spaces can be pushed out and pulled back via quasi-isometries. For hierarchically hyperbolic groups, however, this is not true, as group actions are in general not equivariant with respect to any quasi-isometry.
The next lemma describes how to pull back hierarchically hyperbolic group structures on a group $H$ via quasi-isometric homomorphisms. Recall the definition of glueing hieromorphism (Definition \ref{def: glueing_hieromorphism}).

\begin{lemma}[\bf Pulling back hierarchical structures]\label{lemma: pulling back hierarchical structures}
Let \((G, \mf{S}_G)\) be a hierarchically hyperbolic group and let \(f \colon H \to G\) be a quasi-isometric homomorphism. Then \(H\) can be endowed with a hierarchically hyperbolic structure \(\mf{S}_H\) defined as follows. 
\begin{enumerate}
\item The set \(\mf{S}_H\) coincides with \(\mf{S}_G\), and the associated hyperbolic spaces also coincide. 
\item The projections \(\pi^H_U \colon H \to \fontact U\) are defined as the composition \(\pi^G_U \circ f\), where \(\pi^G_U \colon G \to \fontact U\) is the projection associated to \((G, \mf{S}_G)\). 
\item The relations between the elements of \(\mf{S}_H\) are unchanged, and so are the maps \(\rho_V^U\).
\end{enumerate}
Moreover, \(f\) is a glueing hieromorphism between \(H\) and \(G\).
\end{lemma}
\begin{proof}
Since \(f\) has finite kernel and finite index image, it is clear that \(f\) induces a quasi-isometry. Thus \((H, \mf{S}_H)\) is a hierarchically hyperbolic space.  In order to show that it is a hierarchically hyperbolic group, we now show that the structure induced above is $H$--equivariant. Since \(G\) acts on \(\mf{S}_G\), we obtain that \(H\) acts on \(\mf{S}\) as well via \(f\). Since \(f(H)\) has finite index in \(G\), we obtain that the action has finitely many orbits. 
We now show that every $h\in H$ and $U\in\mathfrak{S}_H$ there exists an isometry $h_U:\mathcal{C}U\to\mathcal{C}hU$ such that the following diagram commutes
\begin{equation}\label{coarsely.commuting.diagrams.hhg}
\xymatrix{
H\ar[r]^{h}\ar[d]_{\pi_U} & H\ar[d]^{\pi_{hU}}\\
\mathcal{C}U\ar[r]_{h_U}&\mathcal{C}hU
}
\end{equation}
Indeed, if we define $h_U$ as the isometry induced by $f(h)$ on $\mathcal{C}U$ we obtain that $h_U\circ\pi^H_U(h')=f(h)^*_U\circ\pi_U^G\circ f(h')=\pi_{hU}^G(f(h)\cdot f(h'))=\pi_{hU}^H(h\cdot h')$ for every $h'\in H$. \end{proof}

\begin{defn}
If \(f \colon H \to G\) is as in Lemma \ref{lemma: pulling back hierarchical structures}, we say that \(\mf{S}_H\) is the \emph{pullback} of the hierarchical structure on \(G\) and denote it by \(f^\ast (\mf{S}_G)\). 
\end{defn}

From the above we immediately obtain the following lemma:

\begin{lemma}\label{lem: functioriality of pullback}
Let \((G, \mf{S})\) be a hierarchically hyperbolic group and let \(H, K\) be groups such that there exist quasi-isometric homomorphisms \(f_1 \colon K \to H\) and \(f_2 \colon H \to G\). Let \(f = f_2 \circ f_1\). Then \(f^\ast \mf{S} = f^\ast_1 \left(f^\ast_2 \mf{S}\right)\), and the map \(f\) 
is a glueing homomorphism. 
\end{lemma}

\subsection{Linearly parametrizable graph of groups}

\begin{defn}\label{def: linearly parametrized}
Let $\mathcal{G}$ be a graph of groups. We say that \(\mathcal{G}\) is \emph{linearly parametrized} if there is a map \(\Phi \colon \pi_1(\mathcal{G}) \to \mathbb{D}_\infty\) such that for each vertex or edge group \(G\), the restriction \(\Phi|_G\) has finite kernel and finite-index image (i.e $\Phi|_G$ is a quasi-isometric homomorphism). 
\end{defn}

\begin{thm}\label{thm: inducing hhg structure}
Let $\mathcal{G}$ be a linearly parametrized graph of groups and let \(G = \pi_1 (\mathcal{G})\).  Then, $G$ admits a hierarchically hyperbolic group structure.
\end{thm}

\begin{proof}
Let \(\Phi \colon G \to \mathbb{D}_{\infty}\) be the map witnessing the linear parametrization of \(G\). 
Equip \(\mathbb{D}_\infty\) with the trivial hierarchically hyperbolic group structure \((\mathbb{D}_\infty, \mf{T})\), where \(\mf{T}\) contains a single element \(T\) and \(\fontact T\) coincides with a Cayley graph for \(\mathbb{D}_\infty\). Endow every vertex \(G_v\) with the pullback structure \((G_v, {\Phi |_{G_v}}^\ast (\mf{T}))\), and endow analogously the edge groups. 
We claim that this turns \(\mc{G}\) into a graph of groups that satisfies the hypothesis of Theorem \ref{comb_thm_ver2}. Since the HHG structure on each vertex group consists of a single element, it satisfies the intersection property and clean containers. Let \(e\) be an edge, \(v\) a  vertex incident to \(e\), and let \(\varphi \colon G_e \to G_v\) be an injective homomorphism. Since both \(G_e\) and \(G_v\) are infinite virtually cyclic, we have that \(\varphi\) is a quasi-isometric homomorphism. Thus, by Lemma \ref{lem: functioriality of pullback}, it induces a glueing hieromorphism. Since \(e\) and \(v\) were generic, the result follows. 
\end{proof}

Thus, from now on we will focus on determining which graphs of 2-ended groups can be linearly parametrized. We begin by showing which amalgams and HNN extensions of linearly parametrizable groups can be linearly parametrized.

\begin{lemma}\label{lemma: amalgam of linearly parametrized}
Let $\mathcal{G}_1$ and \(\mc{G}_2\) be linearly parametrized graphs of groups, and let \(\mathcal{G}\) be a graph of groups obtained connecting \(\mc{G}_1\) and \(\mc{G}_2\) with an edge such that the corresponding edge group is 2-ended. Then, $\mc{G}$ is linearly parametrized.
\end{lemma}
\begin{proof}
Let \(e\) be the added edge and let \(G_e\) be the associated group. We want to show that there are maps \(\Phi_1 \colon \pi_1(\mc{G}_1) \to \mathbb{D}_\infty\) and \(\Phi_2 \colon \pi_1(\mc{G}_2) \to \mathbb{D}_\infty\) that agree on \(G_e\) such that their restriction to vertex/edges subgroups has finite kernel and finite index image. Then the universal property of the amalgamated product yields the desired map \(\Phi \colon \pi_1 (\mc{G}) \to \mathbb{D}_\infty\).

Let \(\Phi_1 \colon \pi_1(\mc{G}_1) \to \mathbb{D}_\infty\) be the function parametrizing \(\mc{G}_1\), and let \(\Phi_2\) be the one for \(\mc{G}_2\). Consider the two restrictions \(\Phi_i|_{G_e}\), for \(i\in \{1,2\}\). Since \(G_e\) is an infinite group by assumption, its image has finite index in the vertex groups adjacent to it. In particular, the restrictions \(\Phi_i|_{G_e}\) have finite kernel and finite index image. By Lemma \ref{lem: Kernel to dihedral is always the same}, we conclude \(\Ker (\Phi_1|_{G_e}) = \Ker (\Phi_2|_{G_e})\). We concentrate now on the images \(\Phi_i(G_e)\) which, by the previous argument, are isomorphic. An infinite index subgroup of the dihedral group has to have the form \(\langle s^k \rangle\) or \(\langle s^k, rs^l\rangle\), for some \(k, l \in \mathbb{Z}-\{0\}\). Suppose that the subgroups \(\Phi_i(G_e)\) have the form \(\langle s^{k_i}, ra^{l_i}\rangle \) respectively (the case where they are both cyclic is analogous). Note that the map \(\rho_l \colon \mathbb{D}_\infty \to \mathbb{D}_\infty \) which sends \(s \to s\) and \(r \to rs^l\) is an isomorphism. Thus, up to postcomposing \(\Phi_i\) with \(\rho_{-l_i}\) we can assume that the images \(\Phi_i(G_e)\) have the form \(\langle s^{k_i}, r\rangle \) respectively. 

Let \(\tau_{k} \colon \mathbb{D}_\infty \to \mathbb{D}_\infty\) be the map that sends \(s \to s^k\) and \(r \to r\). Note that \(\tau_k\) is an injection with finite index image, thus postcomposing with \(\tau_k\) does not alter the fact that a map has finite kernel and finite index image. 
It is now straightforward to verify that the maps \(\Phi_1 : = \tau_{k_2} \circ\Phi_1\) and \(\Phi_2 := \tau_{k_1} \circ \Phi_2\) satisfy the desired requirements.
\end{proof}

A result of this type in HNN extensions does not hold in general, as the following example shows:
\begin{example}
If $H=\langle a\rangle$ is an infinite cyclic group, then it can be linearly parametrized via the map $\Phi:H\to\mathbb{D}_{\infty}$ that sends $a\mapsto r$. Let us construct an HNN extension over $H$ by adding a stable letter $t$ that conjugates $a^2$ to $a^3$. That is to say, $G=H\ast_{ta^2t^{-1}=a^3}$. 

Assume that $\Phi$ can be extended to $\widehat{\Phi}:G\to \mathbb{D}_{\infty}$ that linearly parametrizes $G$. As a consequence we obtain that the relation $\widehat{\Phi}(t)\widehat{\Phi}(a)^2\widehat{\Phi}(t)^{-1}=\widehat{\Phi}(a)^3$ holds in $\mathbb{D}_{\infty}$. As virtually cyclic groups are balanced, $\widehat{\Phi}(t)$ must be trivial. Since $\widehat{\Phi}(a)=\Phi(a)=r$, we obtain as a consequence that $r
^2=r^3$ in $\mathbb{D}_{\infty}$, which is a contradiction. Thus, $\Phi$ cannot be extended to a linear parametrization of $G$.
\end{example}

To determine which HNN extensions of linearly parametrizable groups can be linearly parametrized, we introduce the notion of balanced edge.

\begin{defn}[\textbf{Balanced edge}]\label{def: balanced edges}
Let \(\mc{G}\) be a graph of groups and \(e\) be an edge of \(\mc{G}\). We say that \(e\) is \emph{balanced} if the following holds. 
Let \(\mc{H} = \mc{G}- e\), and let \(\phi_+, \phi_- \colon G_e \to \pi_1(\mc{H})\) be the morphisms associated to \(e\). Then for every infinite order element \(a \in G_e\), if there exists \(h \in \pi_1(\mc{H})\) such that 
\begin{equation}\label{equation: unbalanced edge}
h\phi_{+}(a)^{i}h^{-1}= \phi_{-}(a)^j,
\end{equation}
it follows that \(\vert i \vert = \vert j\vert\).
\end{defn}

\begin{rmk}\label{remark: edges in tree are balanced}
Note that if an edge $e$ in a graph of groups $\mathcal{G}$ is unbalanced then $\pi_1(\mathcal{G})$ is unbalanced. Moreover, by Corollary \ref{corollary: tree product and balance} we have that unbalanced edges can never exist in a graph of groups where the underlying graph is a tree.
\end{rmk}

\begin{lemma}\label{lem: parametrizing HNN}
Let \(\mc{H}\) be a linearly parametrized graph of groups and let \(\mc{G}\) be obtained from \(\mc{H}\) by adding an edge \(e\) with infinite associated edge group. Then \(\mc{G}\) is linearly parametrized if and only if \(e\) is balanced. 
\end{lemma}

\begin{proof}
Let \(A,B\) be the images of the edge group, and let \(\psi \colon A \to B\) be the induced isomorphism. Let \(\Phi\colon H = \pi_1(\mc{H}) \to \dihed\) be the map that linearly parametrizes \(H\). As usual, we use the presentation $\mathbb{D}_{\infty}=\langle r,s\mid srs^{-1}=r^{-1},s^2=1\rangle$. We start by showing that the second condition implies the first. 

Consider the subgroups $\Phi(A), \Phi(B)\leq \mathbb{D}_\infty$. 
Note that every infinite order element of \(A\) has to be sent to \(r^n\) for some \(n \in \mathbb{Z}- \{0\}\). Indeed, those are the only infinite order elements of \(\mathbb{D}_{\infty}\), and since \(\Phi|_{A}\) has finite kernel, infinite order elements cannot be mapped to torsion ones. A similar argument applies for \(B\). 
Thus, \(\Phi(A) \cap \langle r \rangle\) has finite index in \(\langle r\rangle\).

Let $\lvert n\rvert$ and $\lvert m\rvert$ be the index of $\langle\Phi(A)\rangle\cap\langle r\rangle$ in $\langle r\rangle$ and of $\langle\Phi(B)\rangle\cap\langle r\rangle$ in $\langle r\rangle$ respectively. We now show that $\lvert n\rvert=\lvert m\rvert$. Let \(a \in A\) be such that \(\Phi(a)\) generates \(\Phi(A) \cap \langle r\rangle\). Observe that there exists \(h \in H\) and \(i > 0\) such that  $ha^ih^{-1}=\psi(a)^j$, for some \(j >0\). Indeed, since \(\mc{H}\) is linearly parametrized, all its vertices and edges groups are infinite virtually cyclic, and the underlying graph is connected. Thus, \(G_v\) and \(G_w\) are commensurable. By assumption, we need to have \(\lvert i\rvert= \lvert j\rvert\). Thus,  $ha^i h^{-1}=\psi(a)^i$ and, therefore, $\Phi(a)^i=\Phi(\psi(a))^{\pm j}$. By mutiplicativity of index of subgroups we obtain $\lvert\langle\Phi(a)\rangle:\langle r\rangle\rvert=\lvert\langle\Phi(\psi(a))\rangle:\langle r\rangle\rvert$. This shows that \(|n| \leq |m|\). The symmetric argument obtained choosing \(b \in B\) such that \(\Phi(b)\) generates \(\Phi(B) \cap \langle r \rangle\) and considering \(\psi^{-1}(b)\) provides the other inequality. 
Thus \(|n| = |m|\). 

Define a map $\psi':\Phi(A)\to\Phi(B)$  as $\psi'(\Phi(x))=\Phi(\psi(x))$.
By Lemma \ref{lem: Kernel to dihedral is always the same}, \(\mathrm{Ker}(\Phi)\vert_{A} = \mathrm{Ker}(\Phi)\vert_{B}\). Thus, \(\psi'\) is a well defined, injective homomorphism. Since \(\psi\) is surjective, so is \(\psi'\), showing that \(\psi'\) is an isomorphism. Since \(\Phi(a)\)  generates \(\Phi(A) \cap \langle r \rangle\) and \(\psi'(\Phi(a))\) generates \(\Phi(B) \cap \langle r \rangle\), we have \(\Phi(a) = r^m\), \(\Phi(\phi(a)) = r^n\) with \(\vert m \vert = \vert n \vert\).

In particular, \(\Phi\) extends to a homomorphism $\Phi':G\to(\mathbb{D}_{\infty})\ast_{\psi'}$. Consider the presentation  $(\mathbb{D}_{\infty})\ast_{\psi'} = \langle s,r,t\mid srs^{-1}=r^{-1},s^2=1, t\psi'(\Phi(x))t^{-1}=\Phi(x)\quad \forall x\in A\rangle$. Let $\rho:\mathbb{D}_{\infty}\ast_{\psi'}\to\mathbb{D}_{\infty}$ be defined as  $\rho(s)=s,\rho(r)=r$ and $\rho(t)=s^{\lvert n-m\rvert/2\lvert n\rvert}$. Then the map $\widetilde{\Phi}=\rho\circ\Phi':G\to\mathbb{D}_{\infty}$ linearly parametrizes \(G\).

To show that the first condition implies the second one, we argue by contradiction. Consider the presentation \(G = \langle H, t \vert tgt^{-1} = \psi (g), \forall g \in A \rangle\) and assume that for some $h\in H$ and infinite order \(a \in A\) we have $ha^ih^{-1}=\psi(a)^j$ with $\lvert i\rvert\neq\lvert j\rvert$.
Therefore,  $ta^it^{-1}=a^j$. Applying $\widetilde{\Phi}$ we have  $\widetilde{\Phi}(e)\widetilde{\Phi}(a)^i\widetilde{\Phi}(e)^{-1}=\widetilde{\Phi}(a)^j$. However, since $\mathbb{D}_{\infty}$ is virtually cyclic, by Lemma \ref{lemma: finite index and balance} it follows that $\lvert i\rvert$ must be equal to $\lvert j\rvert$, which is a contradiction.
\end{proof}

Combining the above two lemmas we obtain the following.
\begin{corollary}\label{coro: linearly parametr iff all edges balanced}
Let \(\mc{G}\) be a graph of groups with 2-ended vertex and edge groups. Then \(\mc{G}\) is linearly parametrizable if and only if all edges are balanced.
\end{corollary}

\begin{proof}
Assume that $\mc{G}$ is linearly parametrizable by a map $\Phi$ and let $e\in E(\mc{G})$. If $e$ belongs in a spanning tree of $\mc{G}$ then $e$ is a balanced edge by Remark \ref{remark: edges in tree are balanced}. Assume now that $e$ does not belong in a spanning tree. Note first that the subgraph of groups $\mc{G}-e$ of $\mc{G}$ is also linearly parametrizable, as we can use the restricted map $\widetilde{\Phi}=\Phi|_{\pi_1(\mc{G}-e)}$ as linear parametrization. If $e$ is unbalanced, then by Lemma \ref{lem: parametrizing HNN} we obtain that $\widetilde{\Phi}$ cannot be extended to $\pi_1(\mc{G}-e)\ast_{t_e}\cong\pi_1(\mc{G})$, which is a contradiction. Thus, every edge $e$ must be balanced.

To show the converse, let $T$ be a spanning tree in $\mc{G}$. Since every vertex group is 2-ended, we can repeatedly apply Lemma \ref{lemma: amalgam of linearly parametrized} to show that the subgraph of groups $\mc{G}|_T$ is linearly parametrizable. If every edge in $\mc{G}$ is balanced, then we can add one by one the remaining edges in $\mc{G}$ to $T$ and apply Lemma \ref{lem: parametrizing HNN} at each step to obtain the result.
\end{proof}

\subsection{Characterizations of hierarchical hyperbolicity}

With the following lemma, we establish a relation between those graphs of groups that can be linearly parametrized and those which have balanced fundamental group.

\begin{lemma}\label{Lem: unbalanced group iff unbalanced edge}
Let \(\mc{G}\) be a graph of groups with balanced vertex groups. Then \(\pi_1(\mc{G})\) is unbalanced if and only if it contains an unbalanced edge.
\end{lemma}

\begin{proof}
By definition, if \(\mc{G}\) contains an unbalanced edge then \(\pi_1(\mc{G})\) is unbalanced. 
Assume now that $\pi_1(\mc{G})$ is unbalanced. Let $T$ be a spanning tree of the underlying graph $\Gamma$ of $\mc{G}$. Start adding edges in $\Gamma\setminus T$ to $T$ until we obtain a subgraph $\Lambda$ of $\Gamma$ such that $\pi_1(\mc{G}|_{\Lambda})$ is unbalanced and $\pi_1(\mc{G}|_{\Lambda - e})$ is balanced. Split $\pi_1(\mc{G}|_{\Lambda})$ as $\pi_1(\mc{G}|_{\Lambda-e})\ast_{t_e}$, and let \(A, B \in \pi_1(\mc{G}|_{\Lambda-e})\) be the subgroups associated to the HNN extension. By Corollary \ref{corollary: HNN and balance}, there is an infinite order element \(a \in A\) and \(h \in \pi_1(\mc{G}|_{\Lambda-e})\) such that \[ha^{p}h^{-1} = ta^{q}t^{-1},\] for \(\vert p \vert \neq \vert q \vert\), showing that \(e\) is an unbalanced edge.\end{proof}

The final ingredient for the proof of the main theorem of this section is the so-called almost Baumslag--Solitar group, which we now introduce.

\begin{defn}\label{def: almost bs groups}[\textbf{Almost Baumslag--Solitar}]
A group $G$ is called an \emph{almost Baumslag--Solitar group} if it can be generated by two infinite order elements $a,s$ and the relation $sa^is^{-1}=a^j$ holds in $G$ for $i,j\neq 0$. An almost Baumslag--Solitar subgroup is \emph{non-Euclidean} if \(\vert i \vert \neq \vert j \vert\).
\end{defn}

\begin{rmk}\label{remark: almost BS groups}
Almost Baumslag--Solitar groups can look very different from traditional Baumslag--Solitar groups. For instance, any group with presentation $\langle a, b\mid ba^nb^{-1}=a^m, R\rangle$ where $R$ is a non-trivial relator on $\{a,b\}$ that does not forces $a$ nor $b$ to be of finite order is an almost Baumslag--Solitar group. Note, moreover, that an almost Baumslag--Solitar group can be obtained as a quotient of some Baumslag--Solitar group, but such quotient is not, in general, an isomorphism.\end{rmk}

A common theme throughout the rest of the paper will be finding almost Baumslag--Solitar subgroups. Typically, we will find elements \(a, s \in G\) such that \(a\) has infinite order and the relation $sa^is^{-1}=a^j$ holds. When $|i|\neq |j|$, we can immediately conclude that \(\langle a, s \rangle\) is a non-Euclidean Baumslag--Solitar subgroup. Indeed, if $s$ was of finite order then there would exist some $k>0$ such that $s^k=1$. Therefore, as $sa^is^{-1}=a^j$, it follows that $s^ka^{i^k}s^{-k}=a^{j^k}$ contradicting that $a$ is of infinite order.

An interesting question to ask is under which conditions does an almost Baumslag--Solitar group contain $BS(m,n)$ for some $m,n$.
In \cite[Proposition 7.5]{levitt2015quotients} it is shown that if a non-Euclidean almost Baumslag--Solitar group $G$ can be embedded into a GBS group, then $G$ will contain some $BS(m,n)$ for $|m|\neq |n|$.
In \cite[Corollary 9.6]{button2015balanced} it is shown that if a non-Euclidean almost Baumslag--Solitar group $G$ can be embedded into the fundamental group of a graph of torsion-free balanced groups with cyclic edge subgroups then $G$ will contain some $BS(m,n)$ for $|m|\neq|n|$.
Following the same spirit, in Corollary \ref{corollary: main result} we show equivalent conditions under which a non-Euclidean almost Baumslag--Solitar group contains some $BS(m,n)$ for $|m|\neq|n|$.

\begin{corollary}\label{cor: unbalanced implies almost BS}
Let \(\mc{G}\) be a graph of groups containing an unbalanced edge. Then
\begin{enumerate}
    \item \(\pi_1(\mc{G})\) contains a non-Euclidean almost Baumslag--Solitar subgroup;
    \item if $\pi_1(\mc{G})$ is virtually torsion-free then $\pi_1(\mc{G})$ must contain a non-Euclidean Baumslag--Solitar subgroup.
\end{enumerate}
\end{corollary}
\begin{proof}
By definition of balanced edges (Definition \ref{def: balanced edges}), if $e$ is unbalanced and $\phi_{\pm}$ are the monomorphisms associated to the edge $e$, then there exists an infinite order element $a'\in G_e$ and $h\in\pi_1(\mathcal{G}-e)$ such that $h\phi_+(a')^ih^{-1}=\phi_-(a')^j$ for some $|i|\neq |j|$. Let $a$ denote $\phi_+(a')$ and $s$ denote $t_eh$ for short. By assumption, \(a\) has infinite order, and so \(s \neq 1\). Then \(\langle a, s \rangle\) is a non-Euclidean almost Baumslag--Solitar group. 

If, in addition, $\pi_1(\mathcal{G})$ is virtually torsion-free then there exists $N>1$ such that $a
^N$ and $s^N$ belongs in a torsion-free subgroup of $\pi_1(\mathcal{G})$. Note that
\begin{align*}
    s^N a^{N \cdot i^N}s^{-N} & = s^{N-1} (s (a^i)^{N \cdot i^{N-1}} s^{-1}) s^{-(N-1)} = \\
    & = s^{N-1} ( (a^j)^{N \cdot i^{N-1}} ) s^{-(N-1)}= \\
    &= s^{N-2}(s (a^i)^{JN \cdot i^{N-2}} s^{-1}) s^{-(N-2)} = \\
    &= \cdots = a^{N \cdot j^{N}}
\end{align*}
Therefore, the relation $s^N(a^{Ni^N})s^{-N}=a^{Nj^N}$ is satisfied in a torsion-free subgroup $Q$ of $\pi_1(\mathcal{G})$. By Lemma \ref{lemma: subgroups of graph of 2 ended groups}, $Q$ is a generalized Baumslag--Solitar group. Since $Ni^N/Nj^N=(i/j)^N\neq\pm 1$, by \cite[Proposition 7.5]{levitt2015quotients} the subgroup $\langle a^N,s^N\rangle$ contains some non-Euclidean Baumslag--Solitar group.
\end{proof}

Combining Lemma \ref{Lem: unbalanced group iff unbalanced edge} with Corollary \ref{cor: unbalanced implies almost BS} we obtain Theorem \ref{thm: non-euclidean BS iff unbalanced edge intro section} from the introduction:

\begin{thm}\label{thm: non-euclidean BS iff unbalanced edge}
Let \(\mc{G}\) be a graph of groups where none of the vertex groups contain distorted cyclic subgroups. Then \(\pi_1(\mc{G})\) contains a non-Euclidean almost Baumslag--Solitar subgroups if and only if \(\mc{G}\) has an unbalanced edge. \end{thm}

\begin{proof}
If $G=\pi_1(\mc{G})$ contains a non-Euclidean almost Baumslag--Solitar subgroup then it is unbalanced. By Lemma \ref{Lem: unbalanced group iff unbalanced edge} we obtain that $\mc{G}$ must contain some unbalanced edge. Corollary~\ref{cor: unbalanced implies almost BS} shows the converse.
\end{proof}

We are now ready to prove the main result of this section.

\begin{thm}\label{thm: balanced edges and BS}
Let \(\mathcal{G}\) be a graph of groups, where all vertex and edge groups are two-ended. Assume moreover that $\pi_1(\mathcal{G})$ is virtually torsion-free. Then the following are equivalent.
\begin{enumerate}
    \item \(\pi_1(\mc{G})\) admits a hierarchically hyperbolic group structure.
    \item \(\mc{G}\) is linearly parametrizable.
    \item \(\pi_1(\mc{G})\) is balanced.
    \item \(\pi_1(\mc{G})\) does not contain \(\mathrm{BS}(m,n)\) with \(\vert m \vert \neq \vert n\vert\).
     \item \(\pi_1(\mc{G})\) does not contain a distorted infinite cyclic subgroup.
\end{enumerate}
\end{thm}
\begin{proof}

\item[\fbox{$3\Leftrightarrow 2$}] By Corollary \ref{coro: linearly parametr iff all edges balanced} we have that $\pi_1(\mc{G})$ is linearly parametrizable if and only if every edge $e$ in $\mc{G}$ is balanced. Moreover, by Lemma \ref{Lem: unbalanced group iff unbalanced edge} we have that every edge in $\mc{G}$ is balanced if and only if $\pi_1(\mc{G})$ is balanced.

\item[\fbox{$5\Rightarrow 3$}] Assume that $\pi_1(\mc{G})$ is unbalanced. Therefore, by Lemma \ref{Lem: unbalanced group iff unbalanced edge} there is an edge \(e\), an infinite order element \(a \in G_e\) and an element \(h \in \pi_1(\mc{G}- e)\) such that \[h\phi_{+}(a)^{i}h^{-1}= \phi_{-}(a)^j,\] with \(\vert i \vert \neq \vert j \vert\). Let \(x= \phi_+(a)\) and \(y= \phi_{-}(a)\). Since \(e\) is unbalanced, there is a spanning tree that does not contain \(e\). In particular, we can assume there is a stable letter \(t\) associated to the edge \(e\) such that \(tyt^{-1}= x\). We claim that \(\langle x \rangle\) is distorted. Note that \(x\) is of infinite order. To simply notation, we will write \( A \approx^{r} B\) if \(\vert A - B \vert \leq r\). We have:
\begin{align*}
d\left(1, x^{N\cdot i }\right) \approx^{2\vert h \vert} d\left(1, hx^{N \cdot i} h^{-1}\right) = d\left(1, y^{N \cdot j}\right) \approx^{2\vert t \vert} d\left(1, x^{N \cdot j}\right).
\end{align*}
This is to say, for each \(N\) we have \(\left\vert d\left(1, x^{N \cdot i}\right) - d\left(1, x^{N \cdot j}\right) \right\vert \leq 2\left( \vert h \vert + \vert t \vert\right)\). Since \(\vert i \vert \neq \vert j \vert\), it is now a standard argument to show that \(\langle x \rangle \) is distorted.
Indeed, restating the argument before for a general exponent \(M\) we have \(d\left(x^M, x^{\left\lfloor \frac{\vert j \vert}{\vert i \vert} M\right\rfloor}\right) \leq \vert h \vert + \vert t \vert + i\). Assuming that \(\vert i \vert > \vert j \vert\), we can  iterate the inequality above to obtain that \(d(1, X^M)\) is comparable to \(\log_{\frac{\vert j \vert}{\vert i \vert}}(M) \cdot (\vert h \vert + \vert t \vert + i)\). That is to say, \(d(1, X^M)\) grows logarithmically, showing that the map \(n \mapsto x^{n}\) cannot be a quasi-isometric embedding.

\item[\fbox{$4\Rightarrow 3$}] Assume that $\pi_1(\mc{G})$ is unbalanced. Therefore, by Lemma \ref{Lem: unbalanced group iff unbalanced edge}, $\mc{G}$ must contain an unbalanced edge. The second item of Corollary \ref{cor: unbalanced implies almost BS} concludes the proof.

\item[\fbox{$1\Rightarrow 5$}] Follows from \cite[Theorem 7.1]{HHSBoundaries} and  \cite[Theorem 3.1]{ durham2018corrigendum}.

\item[\fbox{$2\Rightarrow 1$}] Follows from Theorem \ref{thm: inducing hhg structure}.

\item[\fbox{$5\Rightarrow 4$}] Since non-Euclidean Baumslag--Solitar groups contain distorted cyclic subgroups if $G$ contains some non-Euclidean Baumslag--Solitar subgroup we obtain the result.

\end{proof}

\begin{thm}\label{thm: balanced edges and BS version 2}
Let \(\mathcal{G}\) be a graph of groups, where all vertex and edge groups are two-ended. Then the following are equivalent.
\begin{enumerate}
    \item \(\pi_1(\mc{G})\) admits a hierarchically hyperbolic group structure.
    \item \(\mc{G}\) is linearly parametrized.
    \item \(\pi_1(\mc{G})\) is balanced.
    \item \(\pi_1(\mc{G})\) does not contain a non-Euclidean almost Baumslag--Solitar subgroup.
    \item \(\pi_1(\mc{G})\) does not contain a distorted infinite cyclic subgroup.
\end{enumerate}
\end{thm}

\begin{proof}

Assume that $\pi_1(\mc{G})$ is unbalanced. Therefore, by Lemma \ref{Lem: unbalanced group iff unbalanced edge}, $\mc{G}$ must contain an unbalanced edge. The first item of Corollary \ref{cor: unbalanced implies almost BS} shows the implication $4\Rightarrow 3$. The rest of the implications are the same as in Theorem \ref{thm: balanced edges and BS}.
\end{proof}

\section{Hierarchical hyperbolicity of hyperbolic-2-decomposable groups}\label{section: graph of word hyperbolic groups}

 The goal of this section is to extend Theorems \ref{thm: balanced edges and BS} and \ref{thm: balanced edges and BS version 2} to (hyperbolic)-2-decomposable groups. In this case, we cannot use linear parametrization to obtain hierarchical hyperbolicity. To solve this issue, the first step is to consider (2-ended)-2-decomposable graphs of groups associated to the various edge groups, called \emph{conjugacy graphs}. Intuitively, a conjugacy graph records if an edge group is involved in the presence of Baumslag--Solitar groups. Indeed, it turns out that if all conjugacy graphs are lineraly parametrized, then the group does not contain non-Euclidean Baumslag--Solitar subgroups. Firstly we will prove a combination theorem  for relatively hyperbolic groups (Theorem  \ref{theorem: readaptation of relative HHG}) that allows us to obtain structures on the vertex groups compatible with the ones on the edge groups.
 
We begin by showing the following lemma. This allows us, without loss of generality, to restrict our attention to graphs of hyperbolic groups with infinite virtually cyclic edge groups.

\begin{lemma}[\bf Dealing with finite vertices/edges]
Let \(\mc{G}\) be a graph of groups such that \(\pi_1(\mc{G})\) is infinite and \(\mc{G}\) has hyperbolic vertex groups and virtually cyclic edge groups. Then there exists a finite graph of groups \(\mc{G}'\) with infinite hyperbolic vertex groups and 2-ended edge groups such that $\pi_1 (\mc{G}')=\pi_1(\mathcal{G})$.
\end{lemma}

\begin{proof}
Given a graph of groups \(\mc{H}\) let \(F(\mc{H})\) be the set of edges with finite associated edge group, that is \(\{e \in E(\mc{H}) \mid \vert G_e \vert \leq \infty\}\). Let \(\mc{G}_0= \mc{G}\). We will produce a sequence of graph of groups \(\mc{G}_i\) such that \(\pi_1(\mc{G}_i) \cong \pi_1(\mc{G})\), \(\mc{G}_i\) has hyperbolic vertex groups and virtually cyclic edge groups and \(\vert F(\mc{G}_i)\vert < \vert F(\mc{G}_{i-1})\vert\). Since the graph of groups is finite, eventually we will find \(\mc{G}_n\) such that \(F(\mc{G}_n) = \emptyset\). In particular, if \(\mc{G}_n\) has at least one edge, then the associated edge group is infinite. Hence, the vertex groups needs to be infinite and we are done.  If there are no edges, then there is a single vertex labelled by \(\pi_1(\mc{G})\), which is hyperbolic by construction. Since, by assumption \(\pi_1(\mc{G})\) is infinite, we are done.

Suppose \(\mc{G}_{i}\) is defined. Firstly, suppose that there is \(e \in F(\mc{G}_{i})\) such that there exists a spanning tree \(T_e\) of \(\mc{G}_{i}\) containing \(e\) (recall that \(\pi_1(\mc{G})\) does not depend on the choice of spanning tree, as pointed out in Remark \ref{rmk: pi_1 of G does not depend on the spannig tree}). Then the subgroup $G_{e^+}\ast_{G_e}G_{e^-}$ is hyperbolic by Theorem \cite[Corollary Section 7]{BestvinaFeighnCombination}. Then let \(\mc{G}_{i+1}\) be defined from \(\mc{G}_{i}\) by replacing the edge \(e\) and the incident vertices by a single vertex with associated group \(G_{e^+}\ast_{G_e}G_{e^-}\), and leaving the other edge maps unchanged. By doing this, we still have hyperbolic vertex groups and virtually cyclic edge groups.

So, suppose that no element of \(F(\mc{G}_i)\) can be included in a spanning tree. This is to say that all elements of \(F(\mc{G}_i)\) are loops. Let \(e \in F(\mc{G}_i)\), and let \(v\) be the vertex incident to it. Then by \cite[Corollary 2.3]{BestvinaFeighnAddendum}, the HNN extesion \(G_v \ast_{G_e}\) is hyperbolic. Then we define \(\mc{G}_{i+1}\) as the graph of groups obtained from \(\mc{G}_i\) by removing the edge \(e\) and changing the vertex group of \(v\) to \(G_v \ast_{G_e}\).
\end{proof}

From now on, whenever we state a result on  a  graph of hyperbolic groups $\mc{G}$ we will always assume that the associated edge groups $G_e$ are virtually cyclic and infinite. In other words, from now on we assume that the groups considered are hyperbolic-2-decomposable.

Given a vertex group $G_v$, one of the main challenges that we have to face in this setting is the fact that the incoming edge groups do not necessarily form an almost-malnormal collection in $G_v$ (Definition \ref{def: almost malnormal collection}). As a consequence, these edge groups may not be geometrically separated so as to include them in the hierarchical hyperbolic structure of $G_v$.
The following theorem solves this problem, and it is pivotal in the proof of the main theorem in this section. We also stress that it is a consequence of \cite[Theorem 9.1]{HHSII}. 

\begin{thm}\label{theorem: readaptation of relative HHG}
Let \(G\) be a group hyperbolic relative to a family of hierarchically hyperbolic groups \(\{(H_i, \mf{S}_i)\}_{i=1}^n\). Suppose that there is a finite family of subgroups \(\{K_\alpha\}_{\alpha\in\Lambda}\) and homomorphisms \(\phi_\alpha: K_\alpha \to G\) such that for each \(\alpha\) there exists \(i\) and \(g \in G\) such that \(\phi_\alpha (K_\alpha)\) has finite index in \(H_i^g\). 
 Finally, suppose that each group \(K_\alpha\) is equipped with a hierarchically hyperbolic structure \(\mf{K}_\alpha\) such that \(\phi_\alpha^{g^{-1}} \colon (K_\alpha,\mf{K}_\alpha) \to (H_i,\mathfrak{S}_i)\) is a glueing hieromorphism. 
 
 Then there is a hierarchically hyperbolic structure \((G, \mf{S})\) on \(G\) such that \(\phi_\alpha\) is a glueing hieromorphism for every $\alpha$. Moreover, if all \((H_i, \mf{S}_i)\) satisfy the intersection property, so does \((G, \mf{S})\), and similarly for clean containers.  
\end{thm}

\begin{proof}
This theorem is an adaptation of \cite[Theorem 9.1]{HHSII}, in which the authors provide an explicit hierarchical hyperbolic structure on \(G\) and prove that it satisfies the hierarchical hyperbolic axioms. We will follow almost verbatim the part of the proof that describes such a structure on \(G\), but we will not verify the axioms since this already appears in \cite[Theorem 9.1]{HHSII}. We will conclude the proof by showing that the maps \(\phi_\alpha\) can be realized as glueing hieromorphisms. 

\textbf{The structure:} For each $i=1,  \ldots,n$ and each left coset of $H_i$ in $G$, fix a representative $gH_i$. Let $g\mathfrak{S}_i$ be a copy of $\mathfrak{S}_i$ with its associated hyperbolic spaces and projections in such a way that there is a hieromorphism $H_i\to gH_i$ equivariant with respect to the conjugation isomorphism $H_i\to H_i^g$. Let $\widehat{G}$ be the hyperbolic space obtained by coning-off $G$ with respect to the peripherals $\{H_i\}$, and let $\mathfrak{S}=\{\widehat{G}\}\cup\bigsqcup_{g\in g}\bigsqcup_i\mathfrak{S}_{gH_i}$. The relation of nesting, orthogonality or transversality between hyperbolic spaces belonging to the same copy $\mathfrak{S}_{gH_i}$ are the same as in $\mathfrak{S}_{H_i}$. Further, if $U,V$ belong in two different copies of different cosets, then we impose transversality between them. Finally, for every $U\in \mathfrak{S}_{gH_i}$ we declare that $U$ is nested into $\widehat{G}$. 

The  projections  are  defined  as  follows: \(\pi_{\widehat{G}} \colon G \to \widehat{G}\) is  the  inclusion,  which  is  coarsely surjective and hence has quasiconvex image.  For each \(U \in \mf{S}_{gH_i}\), let \(\gate_{gH_i} \colon G \to gH_i\) be the closest-point projection onto \(gH_i\) and let \(\pi^G_U =  \pi^{H_i}_U \circ \gate_{gH_i}\), to extend the domain of \(\pi_U\) from \(gH_i\) to \(G\).  Since each \(\pi^{H_i}_U\) was coarsely Lipschitz on \(\fontact U\) with quasiconvex image, and the closest-point  projection in $G$ is  uniformly  coarsely  Lipschitz (Lemma~\ref{lem: cpproj is C. Lipschitz}),  the  projection \(\pi^{G}_U\) is  uniformly  coarsely Lipschitz and has quasiconvex image.  For each \(U,V \in \mf{S}_{gH_i}\), the various \(\rho_U^V\) and \(\rho_V^U\) are already defined. If \(U \in \mf{S}_{gH_i}\)  and \(V \in \mf{S}_{g'H_j}\), then \(\rho_V^U = \pi_V(\gate_{g'H_j}(gH_i))\).  Finally, for \( U \neq \widehat{G}\), we define 
\(\rho^U_{\widehat{G}}\) to be the cone-point over the unique \(gH_i\) with \(U \in \mf{S}_{gH_i}\), and \(\rho_U^{\widehat{G}} \colon \widehat{G} \to \fontact U\) is defined as follows: for \(x \in G\), let \(\rho_U^{\widehat{G}}(x) = \pi^G_U(x)\). If \(x \in \widehat{G}\) is a cone point over \(g'H_j \neq gH_i\), let \(\rho_U^{\widehat{G}}(x) = \rho_U^{S_{g'H_j}}\), where \(S_{g'H_j}\) is the \(\nest\)--maximal element of \(\mf{S}_{g'H_j}\). The cone-point over \(gH_i\) may be sent anywhere in \(\fontact U\).

By \cite[Theorem~9.1]{HHSII}, the construction above endows \((G, \mf{S})\) with a hierarchically hyperbolic group structure.

\textbf{Hieromorphisms:}
Fix \(\alpha\). By assumption there exists \(i\) and \(g \in G\) such that \(\phi_\alpha(K_\alpha) \subseteq H_i^{g}\). Moreover, \(\Phi_\alpha = \phi_\alpha^{g^{-1}}\colon (K_\alpha, \mf{K}_\alpha) \to (H_i, \mf{S}_i)\) is a glueing hieromorphism. Our goal is to show that \(\phi\colon (K_\alpha, \mf{K}_\alpha) \to (G, \mf{S})\) can be equipped with a glueing hieromorpism structure.

To simplify notation we will drop the \(\alpha\) and \(i\) subscript and denote \((K, \mf{K}) = (K_\alpha, \mf{K}_\alpha)\), \(\phi = \phi_\alpha\), \((H, \mf{S}_H  ) = (H_i, \mf{S}_i)\) and so on. 

For every \(V \in \mf{K}\), define \(\phi^\diamondsuit (V) = g \Phi^\diamondsuit(V)\) and \(\phi_V^\ast = g^\ast \circ \Phi^\ast_V\), where \(g^\ast\) is the isometry associated to the multiplication \(g \in G\). By assumption, the maps \(\Phi^\ast_V \colon \fontact V \to \fontact \Phi^\diamondsuit V\) are isometries, and for each \(U \in \mf{S}_H\), the space \(\fontact_H U\) and the space \(\fontact_G gU\) are isometric. 
Thus, the maps \(\phi_V^\ast\) are isometries. 

We need to show that the following two diagrams coarsely commute. 
\begin{equation*}
\xymatrix{
K\ar[r]^{\phi}\ar[d]_{\pi^K_V} & G\ar[d]^{\pi^G_{\phi^\diamondsuit(V)}}\\
\mathcal{C}V\ar[r]_{\phi^*_U}&\mathcal{C}\phi^\diamondsuit(V)
}\qquad\qquad\qquad 
\xymatrix{
\mathcal{C}V\ar[rr]^{\phi^*_V}\ar[d]_{\rho^V_U} && \mathcal{C}\phi^\diamondsuit(V)\ar[d]^{\rho^{\phi^\diamondsuit(V)}_{\phi^\diamondsuit(U)}}\\
\mathcal{C}U\ar[rr]_{\phi^*_U}&&\mathcal{C}\phi^\diamondsuit(U)
}
\end{equation*}
This is a matter of unwinding the definitions. We will check the first one, the second is analogous. So, let \(x\in K\). Recall that \(\phi(x) = g\Phi(x)g^{-1} \in gH_ig^{-1}\). Then
\begin{equation}\label{eq: first eq Theorem on rely hyp HHS}
\begin{split}
        \pi^{G}_{\phi^\diamondsuit(V)}(\phi(x)) =  g^\ast\circ \pi^{H_i}_{\Phi^\diamondsuit(V)} \circ g^{-1}
      = g^\ast \circ \pi_{\Phi^\diamondsuit (V)}^{H_i}(\gate_{gH_i}(\Phi(x) g^{-1})).
     \end{split}
\end{equation}
Note that \(d(\Phi(x)g^{-1}, gH_i) \leq \vert g \vert \). Since all the maps are coarsely Lipschitz, there is a uniform bound between \(\pi_{\Phi^\diamondsuit (V)}^{H_i}(\gate_{gH_i}(\Phi(x) g^{-1}))\) and \(\pi_{\Phi^\diamondsuit (V)}^{H_i}(\Phi(x))\). That is, up to a uniformly bounded error, we can write Equation \ref{eq: first eq Theorem on rely hyp HHS} as 
\begin{equation}\label{eq: second eq Theorem on rely hyp HHS}
     \pi^{G}_{\phi^\diamondsuit(V)}(\phi(x)) = g^\ast \left(\pi_{\Phi^\diamondsuit (V)}^{H_i}(\Phi(x))\right).
\end{equation}
On the other hand, we have 
\begin{equation}\label{eq: third eq Theorem on rely hyp HHS}
    \phi^\ast_V \circ \pi_V^K(x) = g^\ast \left(\Phi^\ast_U \circ \pi_V^K(x) \right).
\end{equation}
Since \(g^\ast\) is an isometry,  Equations \eqref{eq: second eq Theorem on rely hyp HHS} and \eqref{eq: third eq Theorem on rely hyp HHS} give the result. Note that the constant of the coarse commutativity depends on \(g\). However, since there are only finitely many pairs \((K_\alpha, H_i)\), we obtain uniformity. 
Hence, the map \(\phi\) can be equipped with a hieromorphism structure. By construction, the maps \(\phi^\ast_U\) are isometries, and the hieromorphism is full. To see that it has hierarchically quasiconvex image, observe that its image is at finite Hausdorff distance from a peripheral subgroup, hence it is strongly quasiconvex (Lemma~\ref{lem: peripherals are strongly QC}). Then it is hierarchically quasiconvex by Theorem~\ref{thm: RST strong quasiconvex}. \cite[Thorem~6.3]{russellsprianotran:convexity}.

\textbf{Intersection property and clean containers:}
We start by checking clean containers, that is checking that for each \(U \nest T \in \mf{S}\) we have \(U \bot \cont_\perp^T U\). If \(U = \widehat{G}\) there is nothing to check. Hence, assume \(U \in g\mf{S}_i\) and let \(gS_i\) be the \(\nest\)--maximal element of \(g\mf{S}_i\). Recall that the relations on \(\mf{S}\) are defined such that if
 \( U,V \in \mf{S}-\{\widehat{G}\}\) are not transverse, then  there is \(i \in \{1, \dots, n\}\) and \(g \in G\) such that \(U,V \in g\mf{S}_i\). In particular, \(U \bot V\) implies \(U,V \in g\mf{S}_i\). Hence, \(\cont_\perp^{\widehat{G}} U = \cont_\perp^{gS_i} U\). Moreover, if \(U \nest T\) and \(T \neq \widehat{G}\), it follows \(T \in g\mf{S}_i\). Since we assumed that \((H_i, \mf{S}_i)\) has clean containers, we have \(U \bot \cont_\perp^T U\) for all \(T \in g\mf{S}_i\), completing the proof.

 Consider now the intersection property. By hypothesis, for each \(g\mf{S}_i\) the map \(\wedge^{gH_i}\) is defined. Then define \(\wedge \colon (\mf{S} \cup \{\emptyset\}) \times (\mf{S} \cup \{\emptyset\}) \to (\mf{S} \cup \{\emptyset\})\) by considering the symmetric closure of the following: 
 \[U \wedge V = \begin{cases} U & \text{ if } V = \widehat{G}\\
 U \wedge^{gH_i} V & \text{ if } U,V \in g\mf{S}_i \text{ for some } i, g\\
 \emptyset & \text{ otherwise.}\end{cases}\]
 The only property to verify that does not follow directly is that if \(U \in g\mf{S}_i\) and \(V \in g'\mf{S}_j\) with \(g\mf{S}_i \neq g'\mf{S}_j\), then there is no \(W\) nested in both \(U,V\). But if such a \(W\) existed, then it needs to belong to both \(g\mf{S}_i\) and \(g'\mf{S}_j\), a contradiction.\end{proof}

\subsection{Commensurability and conjugacy graph}

In this subsection we extend the results obtained in Section \ref{section: graph of infinite virtually cyclic groups} to the general setting. The key object that will allow us to do this is the conjugacy graph (Definition \ref{definition: conjugacy graph}). This is a graph of groups that, combined with Theorem \ref{theorem: readaptation of relative HHG}, provides vertex groups with a hierarchical hyperbolic structure realizing edge maps as glueing hieromorphisms.

As the  vertex groups in the graphs of groups considered are not 2-ended, the whole graph of groups cannot be linearly parametrized. Moreover, the edge groups do not necessarily embed into vertex groups in an almost-malnormal way. To overcome those problems, we will consider the elementary closure of subgroups. A systematic study of elementary closures of WPD subgroups (which include cyclic subgroups of hyperbolic groups as a special case) is carried out in \cite{DGO}, where the authors show such subgroups need to be hyperbolically embedded in the ambient group. For the sake of self-containment, we recall some useful properties of the elementary closure.

\begin{defn}[{\bf Elementary closure}]
Let $G$ be a group and let $H$ be a subgroup of $G$. We define the \emph{elementary closure} of $H$ in $G$ as the subgroup
\[E_G(H)=\{g\in G\mid \dhaus{(gH,H)}<\infty\}.\]
\end{defn}
\begin{lemma} \label{lem:elementary closure is maximal for commensurability}
Let \(H, K\) be subgroups of \(G\) such that \(H \cap K\) has finite index in both \(H\) and \(K\), then \(K \leq E_G(H)\). 
\end{lemma}
\begin{proof}
Let \(k \in K\) and \(h \in H\). Our goal is to uniformly bound \(d(kh, H)\). Since \(H\cap K\) has finite index in \(H\), there is \(k_0 \in H \cap K\) at uniformly bounded distance from \(h\). Note that \(kk_0\in K\). Since \(H\cap K\) has finite index in \(K\), there is \(h_0 \in H \cap K\) at uniformly bounded distance from \(k k_0\). By triangular inequality, we get a uniform bound on \(d(kh, h_0)\).
\end{proof}
Note that, in general, \(H\) will not have finite index in \(E_G(H)\). A simple example of this is given by considering the subgroup \(\langle a \rangle \) in \(\langle a \rangle \oplus \langle b \rangle \cong \mathbb{Z}^2\). Indeed, in this case we would have \(E_{\mathbb{Z}^2}(\langle a \rangle) = \mathbb{Z}^2\). This is not the case, however, for 2-ended subgroups of hyperbolic groups.
\begin{lemma}[{\cite[Lemma~6.5]{DGO}}]\label{remark: elementarizer}
Let \(G\) be a hyperbolic group and \(H\) be a 2-ended subgroup. Then \(E_G(H)\) is 2-ended.
\end{lemma}
In particular, observe that \(E_G(H)\) has to be the maximal cyclic subgroup containing \(H\). This yields the following useful lemma.

\begin{lemma}\label{remark: commensurability and malnormality}
Let \(H_1, \dots, H_n\) be 2-ended subgroups of a hyperbolic group \(G\). Then \begin{enumerate}
    \item \(H_i\) and \(H_j\) are commensurable in \(G\) if and only if \(E_G(H_i)\) and \(E_G(H_j)\) are conjugate to each other.
    \item $\{E_G(H_1),\ldots,E_G(H_n)\}$ is an almost-malnormal collection if and only if $H_i$ and $H_j$ are non-commensurable for every $i\neq j$;
\end{enumerate}
\end{lemma}
\begin{proof}
Since \(H_i\) has finite index in \(E_G(H_i)\), we have that \(E_G(H_i)\) and \(E_G(H_j)\) are commensurable if and only if \(H_i\) and \(H_j\) are. In particular, this shows one implication. Suppose that \(E_G(H_i)\) and \(E_G(H_j)\) are commensurable. Up to conjugating one of them we have  that \(gE_G(H_i)g^{-1} \cap E_G(H_j)\) has infinite index in both \(gE_G(H_i)g^{-1}\), and \( E_G(H_j)\). By Lemma \ref{lem:elementary closure is maximal for commensurability} we have \(gE_G(H_i)g^{-1} \leq E_G(E_G(H_j)) = E_G(H_j)\) and, by symmetry, \(E_G(H_j) \leq gE_G(H_i)g^{-1}\). Hence, \(E_G(H_i)\) and \(E_G(H_j)\) are conjugate.

For the second item, observe that if \(E_G(H_i)\) and \(E_G(H_j)\) are not commensurable, since they are 2-ended groups it must follow \(\vert E_G(H_i) \cap gE_G(H_j) g^{-1} \vert \leq \infty\) for all \(g \in G\). Hence they are almost-malnormal. 
\end{proof}

We now introduce the conjugacy graph associated to an edge group. 

\begin{defn}[\bf Commensurability class]

Let $G$ be a group and let $\mathcal{P}$ be a collection of 2-ended subgroups of $G$. We denote by $\approx$ the equivalence relation on $\mathcal{P}$ induced by commensurability. That is to say, $P_1\approx P_2$ whenever \(P_1, P_2\) are commensurable (as in Definition~\ref{def: commensurable subgroups}). For each $P\in\mathcal{P}$ we use $\llbracket P\rrbracket$ to denote its commensurability class.
\end{defn}

\begin{defn}[\bf Equivalence class]
Let $\mc{G}$ be a graph of groups with 2-ended edge groups.

Consider the multiset 
\[U=\{\phi_{e^+}(G_e), \phi_{e^-}(G_e)\mid e \in E(\Gamma)\}\]
of all the images of edge groups into vertex groups counted with repetitions.

Let \(\sim_0\) be the relation on \(U\) defined by imposing \(H_1 \sim_0 H_2\) whenever either there exists \(e\) such that \(H_1 = \phi_{e^+}(G_e)\) and \(H_2 = \phi_{e^-}(G_e)\), or \(H_1, H_2\in G_v\) for some \(v\) and \(H_1\approx H_2\) in $G_v$. Extend $\sim_0$ to an equivalence relation $\sim$ on $U$ by taking the transitive closure of $\sim_0$.

For a vertex group \(H\), we denote by \([H]\) its equivalence class with respect to \(\sim\).
\end{defn}

\begin{defn}[\bf Conjugacy graph]\label{definition: conjugacy graph} Let \(\mc{G}\) be a graph of groups with 2-ended edge groups and let $[H]$ be the equivalence class of an edge group in $\mc{G}$. We define the \emph{conjugacy graph} associated to \([H]\) as the graph of groups \(\Delta_{[H]}\) defined as follows. 

For each vertex group \(G_v \in \mc{G}\), let \([H]_v = \{H' \in [H]\mid H'\leq G_v\}\).

{\bf Vertices: }For each vertex \(v\) of the original graph \(\mc{G}\) and commensurability class \(\llbracket K \rrbracket\) of  \([H]_v\), add one vertex \(v_K\) to \(\Delta_{[H]}\). Choose once and for all a representative \(K\in \llbracket K \rrbracket\) and define \(E_{G_v}(K)\) to be the vertex group associated to \(v_K\). 

{\bf Edges: }For each edge \(e \in \Gamma\) such that \(\phi_{e^+}(G_e) \in [H]\), add an edge between \(\llbracket \phi_{e^+}(G_e) \rrbracket\) and \(\llbracket \phi_{e^-}(G_e)\rrbracket\), with associated edge group \(G_e\). To define the edge maps, let \(K\) be the chosen representative of \(\llbracket \phi_{e^+}(G_e) \rrbracket\). Then there is \(h \in G_{e^+}\) such that \(\phi_{e^+}(G_e)^h \subseteq E_{G_{e^+}}(K)\). If \(\phi_{e^+} \colon G_e \to G_{e^+}\) was the edge map of \(\mc{G}\), let the attaching map of \(\Delta_{[H]}\) be defined as \(\phi_{e^+}^h \colon G_e \to E_{G_{e^+}}(K)\). Note that, by Remark \ref{remark: commensurability and malnormality}, this map is well defined.
\end{defn}

\begin{rmk}\label{remark: remark on conjugacy graph}
In this paper, we consider only graphs of groups with 2-ended edge groups. In particular, by Lemma \ref{remark: elementarizer} the vertex groups of the conjugacy graphs are 2-ended.  As the edge groups of the conjugacy graphs are the same as the original edge groups, the conjugacy graphs have 2-ended vertex and edge groups.
\end{rmk}

\begin{example}
Let $\mathbb{F}_2=\langle a,b\rangle$ be the free group of rank $2$ and consider the group $G$ to be $\pi_1(\mathcal{G})=\mathbb{F}_2\ast_{ta^3t^{-1}=ba^2b^{-1}}$. By construction, the splitting of $G$ has one vertex $v$ with associated vertex group $G_v=\mathbb{F}_2$ and one edge $e$ with associated cyclic edge group $G_e$. We now construct the conjugacy graph $\Delta_{[G_e]}$ associated to $[G_e]$. Note first that the images of the single edge group are commensurable in the vertex group, as $b\langle a^3\rangle b^{-1}\cap\langle ba^2b^{-1}\rangle$ is infinite. Thus, there is a single conjugacy class of $[G_e]$ in $\mathbb{F}_2$ and, therefore, a single vertex in $\Delta_{[H]}$. The associated vertex group of $\Delta_{[H]}$ is $bE_{\mathbb{F}_2}(a^2)b^{-1}=b\langle a\rangle b^{-1}$. There is also a single edge group in $\Delta_{[H]}$ with associated edge group equal to the one in $\mathcal{G}$. The associated attaching maps are $\phi_{e^+}$ and $\phi_{e^-}^b$. The conjugacy graph associated to $[G_e]$ results in the group $\langle a\rangle\ast_{ta^2t^{-1}=a^3}$.
\end{example}

In the following two lemmas, we describe how is the linear parametrization in a graph of 2-ended groups extended to the general setting using the conjugacy graph.

\begin{lemma}\label{lemma: conjugacy path and conjugacy graph}
Let $G\cong\pi_1(\mc{G})$ be a graph of hyperbolic groups with 2-ended edge subgroups and let $e$ be an edge in the underlying graph of $\mc{G}$. If $\Delta_{[G_e]}$ denotes the conjugacy graph associated to $[G_e]$, then $e$ is unbalanced in $\mathcal{G}$ if and only if $\pi_1(\Delta_{[G_e]})$ is unbalanced. 
\end{lemma}

\begin{proof}
Assume first that $\mathcal{G}$ contains an unbalanced edge $e$. Therefore, there exists an infinite order element $a\in G_e$ and $h\in\pi_1(\mc{G}-e)$ such that $h\phi_{e^+}(a)^ih^{-1}=\phi_{e^-}(a)^j$ for some $|i|\neq |j|$. By Lemma \ref{lemma: conjugacy path in G} there is a path $e_1,\ldots, e_k$ in the graph of $\mc{G}-e$  with \(A_{e(1)} = G_\alpha, B_{e(k)}  = G_\beta\) such that $B_{e_j}^{h_j}\cap A_{e_{j+1}}$ is non-trivial for every $j=1,\ldots,k-1$ (i.e $E_{G_{e_j
^+}}(B_{e_j})^{h_j}=E_{G_{e_j^+}}(A_{e_{j+1}})$) and elements \(h_0 \in G_\alpha\) and \(h_i \in G_{b(e_i)}\)  satisfying
\begin{equation}\label{equation: conjugacy path 3}
(t_{e_k}h_k\cdots h_1h_0) \phi_{e^+}(a)^i (t_{e_k}h_k\cdots h_1h_0)^{-1}=\phi_{e^-}(a)^j,
\end{equation}
for some $|i|\neq |j|$.

This means that the conjugacy graph $\Delta_{[G_e]}$ splits as $\pi_1(\Delta_{[G_e]}-e)\ast_{t_e}$. Recall that by definition the attaching maps in $\Delta_{[G_e]}$ are defined as conjugates $\phi_{e'^+}^{h_{e'}}$ in $G_{e'^+}$ of the attaching maps $\phi_{e'^+}$ in $\mathcal{G}$. Therefore, since $\phi_{e^+}(g),\phi_{e^-}(g')$ are conjugate in $\pi_1(\mathcal{G})$, following Equation \eqref{equation: conjugacy path 3} we obtain that $\phi_{e^+}(g)^i=\phi_{e^-}(g)^j$ in $\pi_1(\Delta_{[G_e]}-e)$ where $|i|\neq |j|$.

Assume now that, $\pi_1(\Delta_{[G_e]})$ is unbalanced. We can apply Lemma \ref{lemma: conjugacy path in G} to obtain, \begin{equation}\label{equation second implication unbalanced CG implies unbalanced}(h_kt^{\epsilon_k}_{e_k}\cdots h_1t_{e_1}^{\epsilon_1}h_0)a^p(h_kt^{\epsilon_k}_{e_k}\cdots h_1t_{e_1}^{\epsilon_1}h_0)^{-1}=a^q,\end{equation}
for some $|p|\neq |q|$. Here, \(a\) is of infinite order, the various elements \(h_i\) and \(a\) belong to vertex groups and at least one \(\epsilon_i\) is non zero. Our goal is to modify the above equation to obtain an analogous one that holds in \(\pi_1(\mc{G})\). Let \(H_0\) be the vertex group of \(\Delta_{[G_e]}\) that contains \(a\) and let \(H_1\) be the other vertex group adjacent to \(e_1\) in \(\Delta_{[G_e]}\) (possibly, \(H_0 = H_1\)). Let \(x \in H_1\) be such that \((t_{e_1}^{\epsilon_1}h_0)a^p(t_{e_1}^{\epsilon_1}h_0)^{-1} = x\) in \(\pi_1(\Delta_{[G_e]})\). By definition of conjugacy graphs, there are vertex groups \(G_0, G_1\) of \(\mc{G}\) such that \(H_i \leq G_i\). Since the attaching maps in the conjugacy graph are defined as a conjugates of the attaching maps of \(\mc{G}\), there exist \(k_i \in G_i\) such that the following holds in \(\pi_1(\mc{G})\):
\[(k_1t_{e_1}^{\epsilon_1}h_0k_0)a^p(k_1t_{e_1}^{\epsilon_1}h_0k_0)^{-1} = x\]
Let \(y_1 = (k_1t_{e_1}^{\epsilon_1}h_0k_0)\). Proceeding in this way, we find an element \(y_k = y\) of \(\pi_1(\mc{G}-e)\) such that \[ya^{p}y^{-1} = a^q\] with \(\vert p \vert \neq \vert q \vert\), showing that \(e\) is unbalanced in \(\mc{G}\). \end{proof}

\begin{lemma}\label{lemma: 2 implies 1}
Let \(\mc{G}\) be a graph of groups with hyperbolic vertices and 2-ended edge subgroups. Suppose, moreover, that for each edge \(e\) the conjugacy graph \(\Delta_{[G_e]}\) is linearly parametrizable. Then \(\pi_1(\mc{G})\) admits a hierarchically hyperbolic group structure. 
\end{lemma}
\begin{proof}
For each vertex \(v\in V(\mc{G})\) let \(\{e_i\}\) be the set of incoming edges and let \(E(G_{e_i^{+}})\) be the elementary closure of the images of the edge groups in \(G_v\). Choose representatives \(\{ E_i\}\) of the commensurability classes \(\{\llbracket E(G_{e_i^{+}})\rrbracket\}\). Note that, by Remark \ref{remark: commensurability and malnormality}, \(\{ E_i\}\) forms an almost-malnormal collection of subgroups. In particular, \(G_v\) is hyperbolic relative to \(\{E_i\}\) by Theorem~ \ref{thm:bowditch}.

By assumption,  the conjugacy graph $\Delta_{[G_e]}$ associated to $[G_e]$ is linearly parametrizable for every $e$. That is to say, for every edge $e$ there exists $\Phi_{[G_e]}:\pi_1(\Delta_{[G_e]})\to\mathbb{D}_{\infty}^{(e)}$ such that $\Phi_{[G_e]}|_{G_x}:G_x\to\mathbb{D}_{\infty}^{(e)}$ is a quasi-isometry, where \(G_x\) is either a vertex or edge group of \(\Delta_{[G_e]}\). We endow the various groups \(G_x\) with the  hierarchical hyperbolic structure $(G_x,\{\mathbb{D}_{\infty}^{(e)}\})$ as described in Lemma \ref{lemma: pulling back hierarchical structures}. In particular, this allows to equip with a hierarchically hyperbolic group structure every edge group of \(\mc{G}\) and every group \(E_i \leq G_v\) as before. Note that this is well defined. Indeed, suppose that \(e,f\) are edges incoming in \(v\) and \(E(\phi_{e^{+}}(G_e)), E(\phi_{f^{+}}(G_f))\) are conjugate. Then \(e\sim f\) and hence \(E(\phi_{e^{+}}(G_e))\) and \(E(\phi_{f^{+}}(G_f))\) are identified in the conjugacy graph. Thus the hierarchically hyperbolic structure of the representative \(E\) does not depend on choices. Finally, note that since the trivial hierarchically hyperbolic structure on \(\dihed\) satisfies the intersection property and clean containers, so do all the hierarchically hyperbolic structures considered thus far.

Note that we are now in the hypotheses of Theorem~\ref{theorem: readaptation of relative HHG}, allowing us to equip every vertex group with a hierarchically hyperbolic structure \((G_v,\mathfrak{S}_v)\) that turn the edge maps into glueing hieromorphisms $(G_e,\mathfrak{S}_e)\hookrightarrow (G_v,\mathfrak{S}_v)$. Moreover \((G_v, \mf{S}_v)\) satisfy the intersection property and clean containers. Applying Theorem \ref{comb_thm_ver2} we obtain that $\pi_1(\mc{G})$ is a hierarchically hyperbolic group.
\end{proof}

We now show the proof of the main results of the section and the paper.

\begin{corollary}\label{corollary: main result}
Let \(\mc{G}\) be a graph of groups with hyperbolic vertices and 2-ended edge subgroups.  Assume that $G = \pi_1(\mc{G})$ is virtually torsion-free. The following are equivalent:
\begin{enumerate}
    \item $G$ is a hierarchically hyperbolic group;
    \item the conjugacy graph associated to every equivalence class of edges is linearly parametrizable;
    \item $G$ does not contain $BS(m,n)$ for $\lvert n\lvert\neq \lvert m\lvert$;
    \item \(G\) is balanced;
    \item \(G\) does not contain an infinite distorted cyclic subgroup.
\end{enumerate}
\end{corollary}
\begin{proof}

\item[\fbox{$1\Rightarrow 5$}] Follows from 
\cite[Theorem 7.1]{HHSBoundaries} and  \cite[Theorem 3.1]{ durham2018corrigendum}.

\item[\fbox{$5\Rightarrow 4$}] If \(G\) is non-balanced, then by Corollary~\ref{Lem: unbalanced group iff unbalanced edge}, \(\mc{G}\) contains an unbalanced edge and hence a non-Euclidean Baumslag--Solitar subgroup. Since these subgroups contain an infinite distorted subgroup we obtain the implication.

\item[\fbox{$4\Rightarrow 3$}] By definition, a balanced group cannot contain a non-Euclidean Baumslag--Solitar subgroup.

\item[\fbox{$3\Rightarrow 2$}]
Assume that $\Delta_{[G_e]}$ is not linearly parametrizable for some edge $e$. Theorem \ref{thm: balanced edges and BS} implies that there exists an edge $e\in\Gamma\setminus T$ which is unbalanced in $\Delta_{[G_e]}$. Moreover, Lemma \ref{lemma: conjugacy path and conjugacy graph} ensures that there exists an unbalanced edge in $\mathcal{G}$. By lemma \ref{Lem: unbalanced group iff unbalanced edge} we obtain that $G$ must contain some non-Euclidean Baumslag--Solitar group.

\item[\fbox{$2\Rightarrow 1$}] Follows from Lemma \ref{lemma: 2 implies 1}
\end{proof}

\begin{corollary}\label{corollary: main result version 2}
Let \(\mc{G}\) be a graph of groups with hyperbolic vertices and 2-ended edge subgroups. The following are equivalent:
\begin{enumerate}
    \item $G$ is a hierarchically hyperbolic group;
    \item the conjugacy graph associated to every equivalence class of edges is linearly parametrizable;
    \item $G$ does not contain a non-Euclidean almost Baumslag--Solitar group;
    \item \(G\) is balanced;
    \item \(G\) does not contain an infinite distorted cyclic subgroup.
\end{enumerate}
\end{corollary}

\begin{proof}
The implications are the same as in Corollary \ref{corollary: main result}, except for $4\Rightarrow 3$ and $3\Rightarrow 2$, which we now show.

\item[\fbox{$4\Rightarrow 3$}] By definition, a balanced group cannot contain a non-Euclidean almost Baumslag--Solitar group.

\item[\fbox{$3\Rightarrow 2$}] Assume that $\Delta_{[G_e]}$ is not linearly parametrizable for some edge $e$. Since $\Delta_{[G_e]}$ is a graph of 2-ended groups (Remark \ref{remark: remark on conjugacy graph}), Theorem \ref{thm: balanced edges and BS version 2} implies that $\pi_1(\Delta_{[G_e]})$ is unbalanced. Therefore, Lemma \ref{lemma: conjugacy path and conjugacy graph} ensures that there exists an unbalanced edge in $\mathcal{G}$. By Corollary \ref{cor: unbalanced implies almost BS} we obtain that $G$ must contain some non-Euclidean almost Baumslag--Solitar group.
\end{proof}

As a consequence of this we obtain the following corollary that was included in the introduction:

\begin{corollary}\label{corol: corollary to mainT}
Let \(G = H_1 \ast_C H_2\) where \(H_i\) are hyperbolic and \(C\) is 2-ended. Then \(G\) is a hierarchically hyperbolic group. 
\end{corollary}

\begin{proof}
It follows from Lemma \ref{lemma: amalgam and balance} that $G$ is balanced. From the previous Corollary, we obtain the result.
\end{proof}

\bibliography{Bibliography}
\bibliographystyle{alpha}

\end{document}